\documentclass[12pt,reqno]{amsart}

\usepackage{mathtools}
\usepackage{amssymb}
\usepackage{array}
\usepackage{booktabs}
\usepackage{graphicx}
\usepackage{placeins}
\usepackage{float}
\usepackage{xcolor}
\usepackage[
  colorlinks=true,
  linkcolor=blue!55!black,
  citecolor=green!40!black,
  urlcolor=blue!65!black
]{hyperref}
\hypersetup{
  pdftitle={Neutral Returns at High-Order Grazing: Sharp Cyclicity, Physical Codimension, and Weighted Crossover},
  pdfauthor={Haibo Lu},
  pdfsubject={Periodic-orbit bifurcation at a continuous high-order grazing seam},
  pdfkeywords={grazing bifurcation, neutral periodic orbit, cyclicity,
    continuous piecewise-smooth flow, Poincare map, ambient codimension,
    positive-part cap, weighted crossover, confluent divided differences}
}
\usepackage[nameinlink,noabbrev]{cleveref}

\numberwithin{equation}{section}

\newtheorem{theorem}{Theorem}[section]
\newtheorem{proposition}[theorem]{Proposition}
\newtheorem{lemma}[theorem]{Lemma}
\newtheorem{corollary}[theorem]{Corollary}

\theoremstyle{definition}
\newtheorem{definition}[theorem]{Definition}

\theoremstyle{remark}
\newtheorem{remark}[theorem]{Remark}

\DeclareMathOperator{\Cyc}{Cyc}

\DeclareMathOperator{\rank}{rank}
\DeclareMathOperator{\sgn}{sgn}
\newcommand{\R}{\mathbb R}

\newcommand{\bfa}{\mathbf a}
\newcommand{\dd}{\,\mathrm d}

\title[Neutral Returns at High-Order Grazing]{%
  Neutral Returns at High-Order Grazing:\\
  Sharp Cyclicity, Physical Codimension,\\
  and Weighted Crossover}
\author{Haibo Lu}
\address{Shanghai Institute of Technology, Shanghai 201418,
  People's Republic of China}
\email{luhaibo1985@gmail.com}
\thanks{ORCID: 0009-0000-2717-5968}
\date{August 2026}

\begin{document}

\begin{abstract}
We study a neutral periodic orbit of a continuous two-field flow that is
tangent to a switching seam while nearby orbits enter the second field.  We
separate the roles of two governing integers: the neutral-return order
\(n\geq2\) controls orbit count, whereas the even contact order \(\nu\)
controls grazing codimension, passage scale, and transverse crossover.
Continuity factors the branch mismatch as \(X^+-X^-=hW\); an active
excursion of duration \(O(q^{1/\nu})\) therefore produces a correction of
order \(q_+^{1+1/\nu}\).  A parameter-uniform passage theorem and a
cross-seam Hermite zero theorem then give the sharp fixed-stratum cyclicity:
\(n\) or \(n+1\) periodic orbits, according to the sign coupling of
smooth and active terms.  A rank-\(n\) physical unfolding attains the
applicable bound.  Contact-jet incidence gives ambient codimension
\(\nu-2\) for order-\(\nu\) grazing and \(n+\nu-2\) for simultaneous
neutral grazing.
Transverse contact parameters replace the central power by a weighted
positive-part cap whose onset exponent crosses from \(1+1/\nu\) to the
generic quadratic-contact value \(3/2\).  The same sharp bound holds for both
the exact cap and a prepared physical multiwell return, including births,
mergers, and multiple entry--exit pairs.  A value-only \(C^1\)
counterexample shows that finite cyclicity requires event-derivative
information.
For every even \(\nu\geq4\), a closed polynomial family realizes the
prescribed contact and rank conditions and the sharp periodic-orbit
configurations through actual finite-time first-return maps.
\end{abstract}

\maketitle
\clearpage

\begin{center}
\footnotesize
\textbf{Contents at a glance.}
\hyperref[sec:introduction]{1. Introduction};
\hyperref[sec:framework]{2. Scalar reduction};
\hyperref[sec:mixed-seam]{3. Cross-seam zeros};
\hyperref[sec:sharpness]{4. Fixed-stratum cyclicity};
\hyperref[sec:ambient-codimension]{5. Contact strata};
\hyperref[sec:transverse-cap]{6. Cap and crossover};
\hyperref[sec:transverse-cyclicity]{7. Transverse cyclicity};
\hyperref[sec:closed-realization]{8. Closed realization};
\hyperref[sec:conclusion]{9. Discussion};
\hyperref[app:uniform-hermite]{Appendices A--E: technical proofs}.
\end{center}

\section{Introduction: one orbit, two degeneracies, three questions}
\label{sec:introduction}

A continuous piecewise-smooth flow may switch between smooth vector fields
while remaining continuous.  Let \(h\) define the switching seam
\(\Sigma=\{h=0\}\).  The physical field equals \(X^-\) on \(h\leq0\) and
\(X^+\) on \(h\geq0\), with \(X^-=X^+\) on \(\Sigma\).  We study a
periodic orbit tangent to \(\Sigma\) at one point and otherwise contained in
\(h<0\).  Nearby orbits may briefly enter \(h>0\), follow \(X^+\), and
return to \(h<0\), where they follow \(X^-\); this excursion changes the
Poincar\'e map and may create periodic orbits.

The orbit is \emph{neutral}: on a transverse section its return map \(P\)
satisfies \(P(0)=0\), \(P'(0)=1\), and the first nonzero term of
\(P-\mathrm{id}\) has order \(n\geq2\).  A second degeneracy occurs at the
seam, where the contact has even order \(\nu\).  These two integers have
different dynamical roles:
\begin{equation}
  \begin{array}{c@{\quad}l}
    n&\text{controls the sharp number of nearby periodic orbits},\\
    \nu&\text{controls contact codimension and passage scale,}\\
       &\text{as well as the transverse crossover.}
  \end{array}
  \label{eq:intro-organizing-principle}
\end{equation}
We establish this separation by reducing an actual two-field flow to a
scalar zero problem and realizing the sharp result in a closed polynomial
system.

\subsection{A closed quartic example}
\label{sec:intro-quartic-first}

The case \((n,\nu)=(2,4)\) already displays the full mechanism.  Put
\[
  \varrho=x^2+y^2-1,\qquad
  h_2(x,y)=\varrho-(1-x)^2=y^2+2x-2,
\]
and let \(R=(-y,x)\), \(V=(x,y)\).  For
\(\lambda=(\lambda_0,\lambda_1)\), define
\[
  G_\lambda(x,\varrho)
  =\gamma\varrho^2+\lambda_1\varrho+\lambda_0(1-x)^2,
  \qquad \gamma\neq0,
\]
and
\begin{equation}
  X^-_\lambda=R+\frac12G_\lambda V,\qquad
  X^+_\lambda=X^-_\lambda+\frac{\beta}{2}h_2V,
  \qquad \beta\neq0.
  \label{eq:intro-quartic-fields}
\end{equation}
The physical vector field is
\begin{equation}
  X_\lambda=
  \begin{cases}
    X^-_\lambda,&h_2\leq0,\\
    X^+_\lambda,&h_2\geq0.
  \end{cases}
  \label{eq:intro-quartic-physical-field}
\end{equation}
Since \(X^+_\lambda-X^-_\lambda=(\beta/2)h_2V\), the extensions agree on
the seam and the physical field is locally Lipschitz.

\begin{figure}[t]
  \centering
  \includegraphics[width=0.96\textwidth]{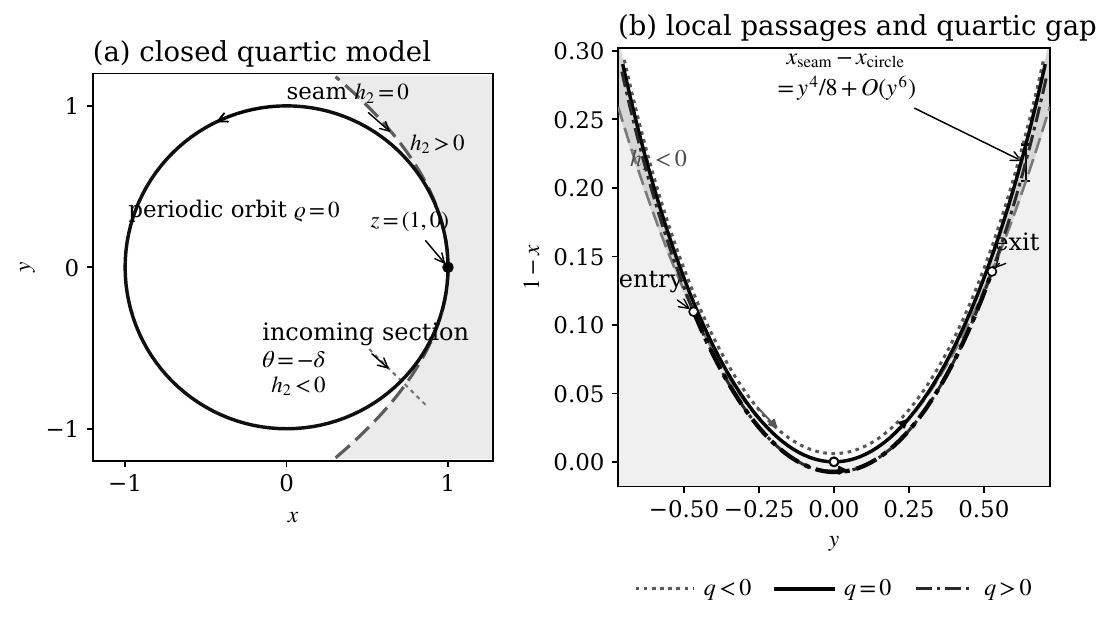}
  \caption{The closed \((n,\nu)=(2,4)\) model.  Panel~(a) shows the central
  orbit, seam, incoming section, and flow direction; shading marks the
  \(X^+\)-side.  In panel~(b), light gray marks the active side \(h_2>0\),
  while darker gray marks the order-four gap between the seam and the
  central circle.  The three line styles distinguish passages with \(q<0\),
  \(q=0\), and \(q>0\).  The \(q>0\) passage includes active entry and
  exit, and the gap satisfies
  \(x_{\rm seam}-x_{\rm circle}=y^4/8+O(y^6)\).  The nearby passages are
  plotted at \(\lambda=0\), \(\gamma=0.35\), \(\beta=0.55\), and
  \(q=\pm0.012\).  The active excursion supplies the
  \(q_+^{5/4}\) term in the local return equation.}
  \label{fig:closed-quartic-geometry}
\end{figure}

At \(\lambda=0\), the unit circle \(\Gamma_0=\{\varrho=0\}\) is a
period-\(2\pi\) orbit and
\[
  h_2|_{\Gamma_0}=-(1-x)^2.
\]
Thus \(\Gamma_0\) stays in \(h_2<0\) except at \(z=(1,0)\), where it has
order-four contact with \(\Sigma=\{x=1-y^2/2\}\).  On an incoming radial
section just before \(z\), let \(q\) label the incoming reference orbit,
with \(q=0\) selecting \(\Gamma_0\).  Here \(q\) is a signed orbit
coordinate, whereas \(t\) below denotes local passage time.  For \(q<0\)
the orbit remains in \(h_2<0\) and follows \(X^-\); at \(q=0\) it is
tangent to the seam; for \(q>0\) it enters \(h_2>0\), follows \(X^+\), and
exits to \(X^-\).  Thus \(X^+\) contributes to the return map through the
active-side excursion.

In reference flow-box time \(t\) centred at \(z\),
\[
  h_2(t,q)=q-\frac14t^4+\text{higher-order terms}.
\]
Entry and exit occur at \(t=O(q^{1/4})\).  Continuity factors the branch
mismatch as \(X^+-X^-=h_2W\), so its size on the active arc is \(O(q)\).
The accumulated correction is therefore
\[
  O(q)\,O(q^{1/4})=O(q^{5/4}).
\]
At even contact order \(\nu\), the same calculation gives active duration
\(O(q^{1/\nu})\) and correction \(O(q^{1+1/\nu})\).  Thus the continuous
flow produces the fractional power directly.

For the quartic family, the time-\(2\pi\) Poincar\'e displacement is
\begin{equation}
  \begin{aligned}
    \Delta(q,\lambda)
      &=S(q,\lambda)+q_+^{5/4}K(q_+^{1/4},q,\lambda),\\
    S(q,0)&=2\pi\gamma q^2+O(q^3),
    &K(0,0,0)&=\frac{8\sqrt2}{5}\beta.
  \end{aligned}
  \label{eq:intro-quartic-family}
\end{equation}
Thus \(c=2\pi\gamma\) records the neutral term and
\(d=8\sqrt2\,\beta/5\) records the transmitted two-field correction.
Zeros of \(\Delta\) correspond to periodic orbits.  The sharp local bound is
two for \(cd>0\) and three for \(cd<0\), coupled across the seam: roots may
lie on either side, and the explicit
chambers in \cref{rem:quartic-chambers} include both
\((N_-,N_+)=(1,2)\) and \((2,1)\).

\subsection{Main results}

For general \(n\geq2\) and even \(\nu\geq2\), the localized passage and
regular return give the full return germ
\begin{equation}
  \begin{aligned}
    \Delta(q,p)
    &=S(q,p)+q_+^{1+1/\nu}K(q_+^{1/\nu},q,p),\\
    c&=\frac{\partial_q^nS(0,0)}{n!}\neq0,
    &d&=K(0,0,0)\neq0,
  \end{aligned}
  \label{eq:intro-full-germ}
\end{equation}
with \(\partial_q^jS(0,0)=0\) for \(0\leq j<n\).  The main conclusions
answer three questions about this orbit.

\begin{enumerate}
\item \emph{How many periodic orbits bifurcate?}  The Cross-Seam theorem
  couples the two sides through their common seam jet and yields
  \begin{equation}
    \Cyc_{(0,0)}(\Delta)
    \leq
    \begin{cases}
      n,&cd>0,\\
      n+1,&cd<0.
    \end{cases}
    \label{eq:intro-main-law}
  \end{equation}
  A rank-\(n\) unfolding attains the applicable bound with hyperbolic
  periodic orbits of alternating stability.

\item \emph{What is the physical codimension?}  Full
  rank of the contact-incidence map gives an order-\(\nu\) grazing stratum
  of ambient codimension \(\nu-2\).  Rank \(n\) of the restricted smooth
  return jet then gives
  \begin{equation}
    \operatorname{codim}\mathcal N_{n,\nu}=n+\nu-2
    \label{eq:intro-total-codimension}
  \end{equation}
  for simultaneous neutral grazing.

\item \emph{What survives transverse contact perturbations?}  The leading
  passage becomes the weighted positive-part cap
  \begin{equation}
    \mathcal C_{\nu,a}(\rho,\mathbf u)
    =\int_{\R}\left(
      \rho-at^\nu-\sum_{j=1}^{\nu-2}u_jt^j
    \right)_+\dd t.
    \label{eq:intro-cap}
  \end{equation}
  Its central exponent is \(1+1/\nu\), while a fixed generic linear tilt
  has the nondegenerate-minimum exponent \(3/2\).  The bound
  \eqref{eq:intro-main-law} holds for the exact cap and for the prepared
  physical multiwell return, including levels at which active intervals are
  born or merge.  A
  value-only \(C^1\) example with the same weighted size estimate has
  arbitrarily many zeros, showing that the derivative structure is
  essential.
\end{enumerate}

The simplest multiwell already explains the event mechanism.  For
\(\mu>0\), let
\[
  \Phi_\mu(t)=t^4-\mu t^2,
  \qquad \min\Phi_\mu=-\frac{\mu^2}{4}.
\]
A level between the minima and the central maximum produces two active
intervals and four seam events; above the maximum the intervals have merged.
For the exact cap, every regular event \(t_j\) contributes
\(1/|\Phi_\mu'(t_j)|>0\) to its second derivative.  The prepared physical
return retains this common sign after the corresponding flow-variation
factors are included.  Multiple wells therefore reinforce the common
derivative monotonicity rather than contributing independent root counts.

\begin{figure}[H]
  \centering
  \includegraphics[width=0.98\textwidth]{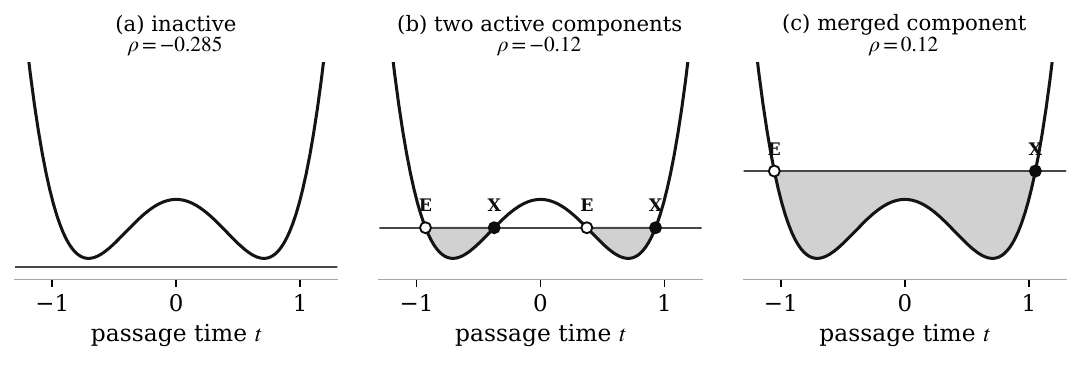}
  \caption{The active set \(\{t:\rho>\Phi_\mu(t)\}\) for
  \(\Phi_\mu(t)=t^4-\mu t^2\), \(\mu=1\).  Raising \(\rho\) creates two
  components and then merges them.  Entries are labeled E and plotted open;
  exits are labeled X and plotted filled.  For the exact cap, all regular boundary terms are positive and
  its first derivative continues across the two critical levels.
  \Cref{thm:prepared-multiwell-boundary-curvature} gives the analogous
  signed continuation for the physical passage.}
  \label{fig:intro-multiwell-events}
\end{figure}
\FloatBarrier

Two reusable results drive the analysis.  The Cross-Seam theorem couples
roots across a shared seam jet; the Prepared Multiwell Boundary-Event
Theorem controls positive-part passages as event topology changes.
For every even \(\nu\geq4\),
\cref{sec:closed-polynomial-realization} realizes their hypotheses in a
polynomial planar family.

\subsection{Relation to earlier entry--exit and grazing results}

Kowalczyk et al.\ \cite{KowalczykEtAl2006} organized local codimension-two
discontinuity-induced bifurcations of limit cycles into a degenerate grazing
point, a nonhyperbolic grazing cycle, and simultaneous grazing events.  For
\(\nu>2\), the present singularity combines the first two mechanisms at
higher codimension; \(\nu=2\) retains the nonhyperbolic-cycle mechanism at
ordinary quadratic grazing.  Its multiple entry--exit pairs arise by
splitting one contact germ rather than from independent grazing points of
the third type.  Colombo and Dercole \cite{ColomboDercole2010} analyzed
generic fold, flip, and Neimark--Sacker cycles through a Poincar\'e map
defined on the noncontacting side.  The common case is \((n,\nu)=(2,2)\):
their fold has multiplier \(+1\) and generic grazing contact.  They derive
the local parameter-plane bifurcation curves from the one-sided return.
Here continuity constructs the physical return on both sides of the seam,
and arbitrary \(n\geq2\) and even \(\nu\geq2\) lead to coupled total
cyclicity, contact codimension, and a transverse multi-event passage.

Entry--exit functions and the cyclicity of slow--fast cycles are developed
in \cite{DeMaesschalckSchecter2016,DeMaesschalckDumortierRoussarie2011}.
The separated slow--fast entry--exit--grazing circuit in
\cite{Lu2026Sharp} has \((n,\nu)=(2,2)\) and obtains a \(q_+^{3/2}\)
correction by composing distinct entry--exit and quadratic-grazing
passages.  The shared minimal-order case consists of the continuous
quadratic-grazing passage, its smooth transmission into a return map, the
sharp two-versus-three scalar count, and rank-two attainment.  The related
preprint's distinct conclusions are the singular-parameter-uniform
entry--exit passage, fixed-positive-\(\varepsilon\) realizations, the full
marked quadratic chamber and stability classification, and the cutoff-Gause
pullback.  The present article instead treats arbitrary \(n\geq2\) and even
\(\nu\geq2\) on one closed neutral orbit, proves the Cross-Seam theorem,
ambient contact codimension, weighted crossover, and prepared physical
multiwell cyclicity, and gives a different closed rank realization for every
even \(\nu\geq4\).  The shared minimal-order statements are proved
independently in the two settings.

Fang and Chen \cite{FangChen2025,FangChen2026GrazingLoops}, together with
Chen, Fang, and Li \cite{ChenFangLi2026Arbitrary}, treat tangent points and
critical or grazing loops of arbitrary or high multiplicity in discontinuous
Filippov settings.  Functional perturbations there generate and count
crossing cycles, sliding loops, and tangent points.  Here branch matching is
continuous, the return is neutral, the unfolding is organized by a finite
contact jet, and the object counted is the total collection of periodic
orbits across the seam.  The present codimension, weighted cap, and physical
multi-event return results address these continuous neutral features.

\subsection{Mathematical context}

Generic impact grazing is classically described by discontinuity maps and
fractional-power normal forms \cite{Nordmark1991}; broader treatments appear in
\cite{diBernardoBuddChampneys2001,
diBernardoBuddChampneysKowalczyk2008,
ColomboDiBernardoHoganJeffrey2012}.  In a continuous stick--slip model with
a first-derivative change across the switching surface, Dankowicz and
Nordmark \cite{DankowiczNordmark2000} derived the quadratic-grazing
\(3/2\) correction.
Nordmark's subsequent theory for prescribed higher-order continuity is the
direct precedent for branch matching
\cite{Nordmark2002HigherContinuity}.  Degenerate impact contacts supply
further high-order geometry
\cite{Chillingworth2010,LimaPerdigao2026}.  Our localized passage treats
every even \(\nu\), derives the full uniform ramified remainder from
\(X^+-X^-=hW\), and transmits it through a regular transition to the actual
first-return map.

Lie-derivative contact strata and polynomial models for traversing flows are
developed by Katz \cite{Katz2017Traversally}.  The incidence argument here
follows a selected finite-parameter family on state--parameter space,
permits \(\nu\) larger than the state dimension, and controls the parameter
projection that produces codimension \(\nu-2\).  The cap derivative is a
one-dimensional instance of coarea \cite[Theorem~3.2.22]{Federer1969}; the
uniform analysis extends across polynomial critical levels, retains all
active components, and respects the contact-jet weighted dilation.

Transverse-event sensitivity provides first-order formulas in regular
chambers
\cite{BurdenSastryKoditschekRevzen2016,KhanBarton2017,
SacconVandeWouwNijmeijer2014}.  In the present family events coalesce as
their slopes vanish, and several active intervals may be born or merge.  The
boundary-event formula supplies the common endpoint sign and continuation
across that finite discriminant.  Finally, Tchebycheff systems and confluent
divided differences give classical zero-count ingredients
\cite{KarlinStudden1966,NovaesTorregrosa2017,deBoor2005}.  The Cross-Seam
theorem combines them with the common seam jet and the full smooth ramified
remainder to obtain one coupled total rather than separate one-sided bounds.

\section{From the flow to a ramified scalar displacement}
\label{sec:framework}
\label{sec:realization}

A continuous two-field contact followed by a regular return produces a
Poincar\'e displacement with a smooth neutral term and an active-side
correction.  We analyze the contact passage and regular transmission before
applying a scalar zero theorem to the resulting one-sided shape.

\subsection{The two-field contact and its signed orbit coordinate}

Let \(p\) range near the origin and let \(h_p\) be a jointly smooth
defining function for a seam
\(\Sigma_p=\{h_p=0\}\) in a two-dimensional phase space.  Let
\(X_p^-\) and \(X_p^+\) be smooth vector-field extensions on a common
neighbourhood, and let
\[
  X_p=
  \begin{cases}
    X_p^-,&h_p\leq0,\\
    X_p^+,&h_p\geq0
  \end{cases}
\]
be the physical field.  We assume \(X_p^+=X_p^-\) on \(\Sigma_p\).
Joint Hadamard factorization then gives a smooth vector field \(W_p\) with
\begin{equation}
  X_p^+-X_p^-=h_pW_p.
  \label{eq:grazing-Hadamard}
\end{equation}
This factorization is the analytic content of continuity at the seam.

Fix an even integer \(\nu\geq2\).  On the selected order-\(\nu\)
contact stratum, suppose there is a smooth contact point \(z_p\) such that
\begin{equation}
  \begin{aligned}
    h_p(z_p)&=0,&
    dh_p(z_p)&\neq0,&
    X_p^-(z_p)&\neq0,\\
    (X_p^-)^jh_p(z_p)&=0 &&(1\leq j<\nu),&
    \kappa_\nu(p)&:=-(X_p^-)^\nu h_p(z_p)>0.
  \end{aligned}
  \label{eq:grazing-geometry}
\end{equation}
The normal strength of the branch mismatch is
\begin{equation}
  \beta(p)=dh_p(z_p)[W_p(z_p)].
  \label{eq:grazing-beta}
\end{equation}
Let \(I_p\) be the local first integral of \(X_p^-\), normalized by
\begin{equation}
  X_p^-I_p=0,\qquad
  I_p(z_p)=0,\qquad
  dI_p(z_p)=dh_p(z_p).
  \label{eq:grazing-first-integral}
\end{equation}
On a small incoming section, use
\(\rho=I_p\) as the signed reference-orbit coordinate.  Negative levels
remain inactive, the zero level touches the seam, and positive levels have
one local active interval.  In the leading flow-box model,
\[
  h_0(t,\rho)=\rho-\frac{\kappa_\nu}{\nu!}t^\nu,
  \qquad
  T_\rho=\left(\frac{\nu!\rho}{\kappa_\nu}\right)^{1/\nu},
\]
so the active interval is \((-T_\rho,T_\rho)\).  Its cap area is
\begin{equation}
  \int_{-T_\rho}^{T_\rho}h_0(t,\rho)\dd t
  =
  \frac{2\nu}{\nu+1}
  \left(\frac{\nu!}{\kappa_\nu}\right)^{1/\nu}
  \rho^{1+1/\nu}.
  \label{eq:leading-grazing-cap}
\end{equation}

\begin{figure}[tbp]
  \centering
  \includegraphics[width=0.82\textwidth]{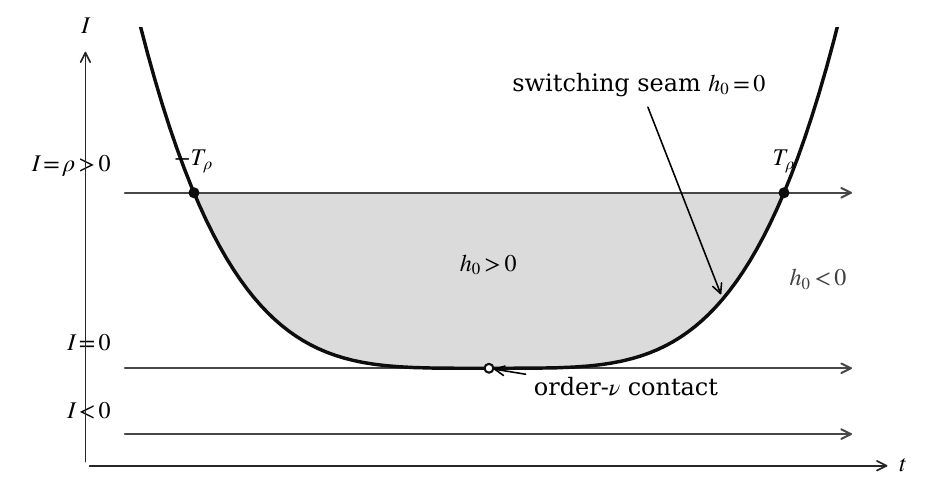}
  \caption{The reference flow-box geometry (shown for \(\nu=4\) and
  \(\kappa_\nu/\nu!=1\)).  Horizontal lines are reference \(X^-\)-orbits.
  A positive level meets the seam at \(t=\pm T_\rho\), and the
  shaded cap is exactly the integral in
  \eqref{eq:leading-grazing-cap}.  The zero level only touches the seam;
  negative levels have no local active interval.}
  \label{fig:grazing-cap}
\end{figure}

\begin{proposition}[Localized Active-Region Passage]
\label[proposition]{thm:localized-active-passage}
\label[proposition]{prop:even-order-grazing-comparison}
Assume
\eqref{eq:grazing-Hadamard}--%
\eqref{eq:grazing-first-integral}, with the contact points and incoming and
outgoing sections depending smoothly on \(p\).  After shrinking a common
grazing box, every orbit with sufficiently small \(\rho>0\) has a unique
first entry and exit, both within time \(O(\rho^{1/\nu})\) of the reference
contact, and no other local seam events.  Let
\(\mathcal D_{\mathrm{gr}}(\rho,p)\) be the outgoing \(I_p\)-coordinate of
the physical excursion minus that of the reference \(X_p^-\)-excursion
launched from the same incoming point.  Then
\begin{equation}
  \mathcal D_{\mathrm{gr}}(\rho,p)
  =
  \begin{cases}
    0,&\rho\leq0,\\
    \rho^{1+1/\nu}\mathcal A_\nu(\rho^{1/\nu},p),&\rho\geq0,
  \end{cases}
  \label{eq:grazing-transition}
\end{equation}
where \(\mathcal A_\nu\) is jointly smooth on its one-sided domain and
\begin{equation}
  \mathcal A_\nu(0,p)
  =
  \mathfrak g_\nu(p)
  :=
  \frac{2\nu}{\nu+1}
  \left(\frac{\nu!}{\kappa_\nu(p)}\right)^{1/\nu}\beta(p).
  \label{eq:grazing-leading-coefficient}
\end{equation}
Equivalently, for a jointly smooth \(B_\nu\),
\begin{equation}
  \mathcal D_{\mathrm{gr}}(\rho,p)
  =
  \mathfrak g_\nu(p)\rho_+^{1+1/\nu}
  +\rho_+^{1+2/\nu}B_\nu(\rho_+^{1/\nu},p).
  \label{eq:grazing-expanded-transition}
\end{equation}
\end{proposition}

The proof uses localization and weighted rescaling in the flow-box equation
\[
  h_p(t,I)
  =
  I-\frac{\kappa_\nu(p)}{\nu!}t^\nu
  +O(t^{\nu+1})+O(tI)+O(I^2).
\]
The scaling \(I=r^\nu+r^{\nu+1}w\), \(t=ru\) confines the seam events and
produces the cap integral \eqref{eq:leading-grazing-cap};
\cref{app:passage} proves the common domain and full uniform remainder.

\subsection{Regular transmission to an actual return}

The local passage forms one segment of a periodic itinerary.  To describe
the effect of a smooth outgoing transition, let
\(\mathcal R_p\) be a smooth local diffeomorphism and define
\[
  \mathcal T_{\mathcal R}[H](q,p)
  =
  \mathcal R_p(q+H(q,p))-\mathcal R_p(q).
\]
This operation \emph{postcomposes a local correction with a regular
transition} and differs from conjugation by a change of Poincar\'e section.

\begin{proposition}[Fixed ramified transmission and first-return recognition]
\label[proposition]{thm:smooth-regular-transmission}
\label[proposition]{prop:reference-return-composition}
Let \((q,p)\mapsto\mathcal R_p(q)\) be jointly \(C^\infty\), suppose
each \(\mathcal R_p\) is a local diffeomorphism,
\(\mathcal R_0(0)=0\), and put
\[
  S(q,p)=\mathcal R_p(q)-q,\qquad
  L_0=(\mathcal R_0)'(0)\neq0.
\]
If the incoming coordinate is \(q=\rho\) in
\cref{prop:even-order-grazing-comparison}, then
\begin{equation}
  P_p(q)
  =
  \mathcal R_p\bigl(q+\mathcal D_{\mathrm{gr}}(q,p)\bigr)
  \label{eq:reference-physical-composition}
\end{equation}
has displacement
\begin{equation}
  \Delta(q,p)
  =
  S(q,p)+q_+^{1+1/\nu}
  K_\nu(q_+^{1/\nu},q,p),
  \label{eq:geometric-full-ramified-form}
\end{equation}
where \(K_\nu\) is jointly smooth and
\begin{equation}
  K_\nu(0,0,0)
  =
  L_0\frac{2\nu}{\nu+1}
  \left(\frac{\nu!}{\kappa_\nu(0)}\right)^{1/\nu}\beta(0).
  \label{eq:composition-leading-coefficient}
\end{equation}
If the local passage and \(\mathcal R_p\) are compatible finite-time flow
maps along a common itinerary, all intervening sections remain transverse,
and the composed orbit has no earlier incoming-section intersection, then
\(P_p\) is the actual first-return Poincar\'e map.
\end{proposition}

\begin{proposition}[Weighted-value transmission]
\label[proposition]{prop:weighted-value-transmission}
Let \((q,p)\mapsto\mathcal R_p(q)\) be jointly \(C^\infty\), suppose each
\(\mathcal R_p\) is a local diffeomorphism, \(\mathcal R_0(0)=0\), and
put \(L_0=(\mathcal R_0)'(0)\neq0\).  Let \(r\downarrow0\),
\(q_r=r^\nu\rho\), and let the remaining parameters follow a weighted
path \(p_r\) with \(\lVert p_r\rVert=O(r^2)\).  Uniformly for
\((\rho,\mathbf u)\) in a fixed compact set, suppose
\[
  H(q_r,p_r)
  =
  \beta_0r^{\nu+1}\mathcal C(\rho,\mathbf u)+O(r^{\nu+2}).
\]
Then
\[
  \mathcal T_{\mathcal R}[H](q_r,p_r)
  =
  L_0\beta_0r^{\nu+1}\mathcal C(\rho,\mathbf u)+O(r^{\nu+2})
\]
uniformly on that compact set.  This proposition controls function values;
derivative and zero-count estimates require separate arguments.
\end{proposition}

\begin{proposition}[Threshold-curvature transmission across finite event strata]
\label[proposition]{prop:threshold-curvature-transmission}
Let \((q,p)\mapsto\mathcal R_p(q)\) be jointly \(C^2\), suppose each
\(\mathcal R_p\) is a local diffeomorphism, \(\mathcal R_0(0)=0\), and put
\(L_0=(\mathcal R_0)'(0)\neq0\).  Let \(\tau\) be continuous with
\(\tau(0)=0\).  Fix \(\rho_0>0\) so that
\(|\tau(p)|<\rho_0\) on the parameter neighbourhood under consideration.
Suppose \(H\) is \(C^1\) in \(q\), with \(H\) and \(H_q\) jointly continuous
in \((q,p)\), and has zero value and \(q\)-derivative on
\(q\leq\tau(p)\).  For each fixed \(p\), let
\(E_p\subset[\tau(p),\rho_0)\) be finite, possibly containing the threshold,
and assume that \(H\) is \(C^2\) on the open inactive side and at every
active level outside \(E_p\).  Fix \(\beta_0\neq0\), and suppose, uniformly
  along the regular active germ as \((q,p)\to(0,0)\),
\[
  H,H_q\longrightarrow0,\qquad
  \sgn H_{qq}=\sgn\beta_0,\qquad
  |H_{qq}|\longrightarrow\infty .
\]
Then, after shrinking, the derivative of
\(\mathcal T_{\mathcal R}[H]\) is continuous on the active-level interval
\([\tau(p),\rho_0)\), and
\[
  q\longmapsto
  \sgn(L_0\beta_0)\,
  \partial_q\mathcal T_{\mathcal R}[H](q,p)
\]
is strictly increasing there.  At regular active points its second
derivative has sign \(\sgn(L_0\beta_0)\), and its magnitude diverges as
\((q,p)\to(0,0)\).  The derivative monotonicity extends across finite births
or mergers, including points where the second derivative is undefined.
\end{proposition}

The exact divided-difference proofs are in \cref{app:passage}.  In the
curvature statement, the divergent term \(L_0H_{qq}\) dominates the bounded
transition terms; \cref{app:boundary-events} joins regular intervals through
continuity of the first derivative.

\subsection{The scalar class and its dynamical interpretation}

For \(q\in\R\), put \(q_+=\max\{q,0\}\).  Let \(p\) range near
\(0\in\R^m\), and set
\begin{equation}
  \alpha_\nu=1+\frac1\nu,
  \qquad
  \gamma_{n,\nu}=\frac{n}{\alpha_\nu}
  =\frac{n\nu}{\nu+1}.
  \label{eq:ramified-exponents}
\end{equation}

\begin{definition}[Order-\((n,\nu)\) ramified displacement]
\label[definition]{def:ramified-displacement}
Fix \(n\geq2\).  An order-\((n,\nu)\) ramified displacement is a germ
\begin{equation}
  \Delta(q,p)
  =
  S(q,p)+q_+^{\alpha_\nu}K(q_+^{1/\nu},q,p),
  \label{eq:full-displacement}
\end{equation}
where \(S\) is jointly \(C^{n+1}\), \(K(s,q,p)\) is jointly \(C^2\)
on its one-sided domain \(s\geq0\), and
\begin{equation}
  \begin{aligned}
    \partial_q^jS(0,0)&=0 &&(0\leq j<n),\\
    c&=\frac{\partial_q^nS(0,0)}{n!}\neq0,
    &d&=K(0,0,0)\neq0.
  \end{aligned}
  \label{eq:central-hypotheses}
\end{equation}
Define
\begin{equation}
  a_j(p)=\frac{\partial_q^jS(0,p)}{j!},
  \qquad
  A(p)=(a_0(p),\ldots,a_{n-1}(p)).
  \label{eq:coefficient-map}
\end{equation}
The sharpness condition is
\begin{equation}
  m\geq n,\qquad \rank DA(0)=n.
  \label{eq:rank-n}
\end{equation}
\end{definition}

The following is the scalar fixed-point version of the standard local
cyclicity of a limit-periodic set; see \cite{Roussarie1998}.
\begin{definition}[Local cyclicity]
\label[definition]{def:cyclicity}
Fix a representative of the displacement germ on
\(({-}\rho_*,\rho_*)\times V_*\), where \(V_*\) is a parameter
neighbourhood of the origin.  For an interval \(U\), let
\(N_{\mathrm{iso}}(\Delta(\cdot,p);U)\) denote the number of distinct
isolated real zeros of \(q\mapsto\Delta(q,p)\) in \(U\), with values in
\(\mathbb N_0\cup\{\infty\}\).  The local cyclicity is
\begin{equation}
  \Cyc_{(0,0)}(\Delta)
  =\inf_{\substack{0<\rho\leq\rho_*\\ 0\in V\subset V_*}}
  \ \sup_{p\in V}
  N_{\mathrm{iso}}\bigl(\Delta(\cdot,p);(-\rho,\rho)\bigr),
  \label{eq:formal-local-cyclicity}
\end{equation}
where \(V\) ranges over parameter neighbourhoods of the origin contained
in \(V_*\), and the infimum is taken in
\(\mathbb N_0\cup\{\infty\}\).  Equivalently, it is the
least eventual uniform bound as the state and parameter neighbourhoods
shrink.  A root at \(q=0\) is counted once.  When
\(P_p(q)=q+\Delta(q,p)\) is an actual first-return map, these zeros are
fixed points of \(P_p\) and therefore periodic orbits of the flow.
\end{definition}

\paragraph{Increasing section coordinates.}
Let the old coordinate \(q\) and a new coordinate \(\widetilde q\) be
related by a jointly smooth based local diffeomorphism
\begin{equation}
  q=\chi_p(\widetilde q),\qquad
  \chi_p(0)=0,\qquad a:=\partial_{\widetilde q}\chi_0(0)>0,
  \label{eq:increasing-section-coordinate}
\end{equation}
and conjugate the return map by
\(\widetilde P_p=\chi_p^{-1}\circ P_p\circ\chi_p\).  Conjugate the
smooth reference map \(P_p^-(q)=q+S(q,p)\) by the same change and use its
displacement as \(\widetilde S\).  If the central displacement has
coefficients \(c,d\) as in
\eqref{eq:central-hypotheses}, then the conjugate displacement has the
same orders \((n,\nu)\) and
\begin{equation}
  \widetilde c=c\,a^{n-1},
  \qquad
  \widetilde d=d\,a^{1/\nu}.
  \label{eq:section-coordinate-coefficient-law}
\end{equation}
Consequently \(\sgn(cd)\), local cyclicity, simplicity and stability of
fixed points, and the rank of the coefficient map are invariant.  If
\(P_p\) is an actual first-return map, then so is \(\widetilde P_p\).
Indeed, at the central parameter, substitute
\(q=a\widetilde q+O(\widetilde q^2)\) into the displacement and apply
\(\chi_0^{-1}\), whose derivative at the origin is \(a^{-1}\).  The
smooth term therefore acquires the factor \(a^{n-1}\), whereas the active
term acquires \(a^{1+1/\nu-1}=a^{1/\nu}\).  For jets of order below
\(n\), parameter-dependent terms from the conjugacy form a triangular
transformation with positive powers of \(a\) on the diagonal; hence the
coefficient rank is unchanged.  Conjugacy gives a bijection of fixed points, preserves
their multipliers, and changes only the coordinate used to record the same
first intersection.

Here \(S,c\) arise from the reference \(X^-\)-return, whereas \(K,d\)
record the active \(X^+\)-excursion.  Thus \(n\) is the neutral-return order
and \(\nu\) is the seam-contact order, as summarized in
\eqref{eq:intro-organizing-principle}.

The polynomial leading family used in the attainment argument is
\begin{equation}
  F_{\mathbf a}(q)
  =
  P_{\mathbf a}(q)+dq_+^{1+1/\nu},
  \qquad
  P_{\mathbf a}(q)=\sum_{j=0}^{n-1}a_jq^j+cq^n.
  \label{eq:canonical-family}
\end{equation}
This family provides the model used to establish attainment for the full
remainder class \eqref{eq:full-displacement}.

\paragraph{Fixed-stratum actual-return hypotheses \(\mathrm{(FR)}_\nu\).}
A family \(P_p\) satisfies \(\mathrm{(FR)}_\nu\) if the hypotheses of
\cref{prop:even-order-grazing-comparison} hold on a common state--parameter
germ with incoming coordinate \(q=\rho\), and a jointly smooth family of
local diffeomorphisms \(\mathcal R_p\) completes the passage by
\(P_p(q)=\mathcal R_p(q+\mathcal D_{\mathrm{gr}}(q,p))\).  Both factors must
be compatible finite-time flow maps along one itinerary, all intervening
sections must remain transverse, and no orbit may return to the incoming
section earlier.  Then \(P_p\) is the actual first-return map on the germ.
The local passage theorem establishes the dynamics within the selected
grazing box.  The common
itinerary, transverse intermediate sections, and no-earlier-return
condition are global hypotheses in \(\mathrm{(FR)}_\nu\); they are
verified explicitly for the closed family in \cref{sec:closed-realization}
for every even \(\nu\geq4\).

\begin{theorem}[Dynamics-to-Scalar Reduction Theorem]
\label{thm:dynamics-to-scalar-reduction}
\label{cor:fixed-order-circuit-recognition}
Let \(P_p\) satisfy \(\mathrm{(FR)}_\nu\), put
\[
  S(q,p)=\mathcal R_p(q)-q,
  \qquad L_0=(\mathcal R_0)'(0),
\]
and suppose
\[
  \partial_q^jS(0,0)=0\quad(0\leq j<n),\qquad
  c=\frac{\partial_q^nS(0,0)}{n!}\neq0,\qquad
  L_0\beta(0)\neq0.
\]
Then its displacement belongs to
\cref{def:ramified-displacement}, with \(d\) given by
\eqref{eq:composition-leading-coefficient}.  Consequently, its local
periodic-orbit cyclicity obeys \cref{thm:sharp-cyclicity}.  If the
coefficient map has rank \(n\), the applicable bound is attained, within
the same \(\mathrm{(FR)}_\nu\) family, by simple periodic orbits
with alternating stability.
\end{theorem}

\begin{proof}
\Cref{prop:even-order-grazing-comparison,%
prop:reference-return-composition} establish the full ramified form and its
nonzero leading coefficient.  Apply
\cref{thm:sharp-cyclicity,cor:stability}.  Because the maps are actual
finite-time first returns, their simple fixed points correspond to periodic
orbits.
\end{proof}

\section{A reusable cross-seam zero theorem}
\label{sec:mixed-seam}

The scalar obstruction used below is independent of its Poincar\'e-map
origin.  It couples a finite-order shape condition on one side of a marked
point to strict convexity on the other.  Agreement of the value and first
derivative at the seam sharpens the sum of the two one-sided bounds.

Fix \(\ell,r>0\).  Let
\begin{equation}
  f\in C^1([-\ell,r]),
  \qquad
  f|_{[-\ell,0]}\in C^n([-\ell,0]),
  \label{eq:abstract-regularity}
\end{equation}
and suppose that there are signs \(\eta,\sigma\in\{-1,1\}\) such that
\begin{align}
  \eta f^{(n)}(x)&>0
  &&(-\ell\leq x\leq0),
  \label{eq:left-shape}\\
  q&\longmapsto \sigma f'(q)
  &&\text{is strictly increasing on }[0,r].
  \label{eq:right-shape}
\end{align}
Write \(N_-(f)\) and \(N_+(f)\) for the numbers of distinct zeros in
\([-\ell,0)\) and \((0,r]\), and let \(z_0(f)=1\) if \(f(0)=0\) and
\(z_0(f)=0\) otherwise.  A zero at the seam is counted once, in accordance
with the distinct-root convention for this finite-regularity setting.

\begin{theorem}[Cross-Seam Hermite Zero Theorem]
\label{thm:cross-seam-hermite-zero}
\label{thm:mixed-seam}
Under \eqref{eq:abstract-regularity}--\eqref{eq:right-shape},
\begin{equation}
  N_-(f)+z_0(f)+N_+(f)
  \leq
  \begin{cases}
    n,&\eta\sigma=1,\\
    n+1,&\eta\sigma=-1.
  \end{cases}
  \label{eq:mixed-seam-bound}
\end{equation}
Every zero is isolated.
\end{theorem}

\paragraph{Coupling at the seam.}
Separate Rolle estimates give upper bounds of \(n\) zeros on the closed left
branch and two on the open right branch.  The matching conditions prevent
independent saturation of these bounds.  The key coupled calculation is the
following confluent Hermite obstruction.  If
\(x_1<\cdots<x_{n-1}<0\) are left-hand zeros and
\[
  a=f(0),\qquad b=f'(0),\qquad
  L=\sum_{i=1}^{n-1}\frac1{-x_i},
\]
then
\begin{equation}
  \sgn(b-La)=\eta.
  \label{eq:confluent-sign}
\end{equation}
Indeed, let \(H\) be the degree-\(n\) Hermite interpolant through those
zeros and through the value and derivative at the seam.  Its leading
coefficient is
\[
  \frac{b-La}{\prod_i(-x_i)}.
\]
Repeated Rolle identifies this coefficient with \(f^{(n)}(\xi)/n!\) for some
\(\xi\in(x_1,0)\).
This calculation provides the cross-seam coupling.  The remaining
sign cases compare \eqref{eq:confluent-sign} with the value and slope forced
by one or two active-side zeros.  Appendix~\ref{app:uniform-hermite} gives
the finite-regularity divided-difference lemmas and the complete proof.

The following form applies to transverse grazing, where the seam in the
section coordinate moves with parameters.

\begin{corollary}[Moving-threshold cyclicity]
\label[corollary]{cor:moving-threshold-cyclicity}
Let \(F(q,\lambda)\) be a family on a fixed interval \((-\rho,\rho)\), and
let \(\tau(\lambda)\in(-\rho/2,\rho/2)\) be continuous.  Suppose that for
every sufficiently small \(\lambda\):
\begin{enumerate}
\item \(F(\,\cdot\,,\lambda)\) is \(C^1\) across \(q=\tau(\lambda)\)
  and is \(C^n\) on the inactive side;
\item a fixed sign \(\eta\) satisfies
  \(\eta\partial_q^nF>0\) for \(q\leq\tau(\lambda)\);
\item for a fixed sign \(\sigma\), the function
  \(q\mapsto\sigma\partial_qF(q,\lambda)\) is strictly increasing for
  \(q\geq\tau(\lambda)\).
\end{enumerate}
Then the number of distinct zeros in \((-\rho,\rho)\) is at most \(n\) if
\(\eta\sigma=1\) and at most \(n+1\) if \(\eta\sigma=-1\).  Every such
zero is isolated.  The conclusion is uniform over the stated parameter
neighbourhood.
\end{corollary}

\begin{proof}
For each \(\lambda\), translate \(q\) by \(\tau(\lambda)\) and apply
\cref{thm:cross-seam-hermite-zero} on the resulting two-sided interval.
The same \(\rho\), parameter neighbourhood, and signs give the uniform
count.
\end{proof}

\subsection{The ramified class satisfies the two shape hypotheses}

For the full remainder in \cref{def:ramified-displacement}, the required
active-side convexity is most transparent in the ramified coordinate
\(s=q^{1/\nu}\).  The assumption \(K\in C^2\) supplies the uniform
first-derivative control used below.

\begin{proposition}[Uniform one-sided shape control]
\label[proposition]{prop:shape-control}
For every displacement in \cref{def:ramified-displacement}, there are
\(\rho>0\) and a parameter neighbourhood \(V\) such that
\begin{align}
  \sgn\partial_q^nS(q,p)&=\sgn c
  &&(-\rho\leq q\leq0,\ p\in V),
  \label{eq:negative-uniform-sign}\\
  \sgn\partial_s\!\left[\Delta_q(s^\nu,p)\right]&=\sgn d
  &&(0\leq s\leq\rho^{1/\nu},\ p\in V).
  \label{eq:positive-uniform-sign}
\end{align}
At \(s=0\), the derivative in \eqref{eq:positive-uniform-sign} is the
right derivative.
\end{proposition}

The calculation is recorded in Appendix~\ref{app:uniform-hermite}.  Its
decisive limit is
\[
  \partial_s\!\left[\Delta_q(s^\nu,p)\right]_{(s,p)=(0,0)}
  =\frac{\nu+1}{\nu}d.
\]
Thus the active derivative \(\sgn(d)\Delta_q\) is strictly increasing.
On the inactive side, \(\sgn(c)S^{(n)}>0\).  The Cross-Seam Hermite Zero
Theorem therefore applies directly to the full class of permitted smooth
remainders.

\section{Sharp fixed-stratum cyclicity}
\label{sec:sharpness}

We now apply the scalar theorem to the return class.  The upper bound is a
shape statement for the entire allowed remainder class.  Its attainment is
a separate rank-lifting statement, and it produces roots on two genuinely
different scales.

\begin{theorem}[Sharp fixed-stratum cyclicity]
\label{thm:sharp-cyclicity}
Every displacement in \cref{def:ramified-displacement} satisfies
\begin{equation}
  \Cyc_{(0,0)}(\Delta)
  \leq
  \begin{cases}
    n,&cd>0,\\
    n+1,&cd<0.
  \end{cases}
  \label{eq:sharp-cyclicity}
\end{equation}
The bound counts each isolated root once, including a multiple root or a
root at the seam, and all local zeros are isolated.  If
\eqref{eq:rank-n} holds, then the applicable bound is attained by simple
roots.  More precisely, for every \(0<r_1<\cdots<r_n\) there is a
parameter path \(p_t\to0\) with \(n\) negative roots
\begin{equation}
  q_j(t)=t(-r_j+o(1)),
  \qquad 1\leq j\leq n.
  \label{eq:negative-sharp-roots}
\end{equation}
If \(cd<0\), the same path also has a positive root
\begin{equation}
  q_{\mathrm{ram}}(t)
  \sim
  \left(
    \left|\frac cd\right|\prod_{j=1}^n r_j
  \right)^{\nu/(\nu+1)}
  t^{n\nu/(\nu+1)}.
  \label{eq:positive-sharp-root}
\end{equation}
\end{theorem}

\begin{remark}[Scope of the sharpness statement]
Establishing the universal total bound, constructing an attaining family,
and classifying all possible left/seam/right root distributions are three
distinct tasks.  The theorem completes the first and, under the rank
condition, the second.  Its witness has distribution \((n,0)\) when \(cd>0\) and
\((n,1)\) when \(cd<0\), where the entries count negative and positive
roots.  Other distributions may realize the same total; their classification
is a separate discriminant problem.
\end{remark}

\begin{remark}[Two inequivalent quartic chambers]
\label[remark]{rem:quartic-chambers}
The distributions \((N_-,N_+)=(1,2)\) and \((2,1)\) both occur when
\(n=2\) and \(cd<0\).  In the canonical family
\[
  F(q)=a_0+a_1q+q^2-q_+^{5/4},
\]
take \(a_1=\varepsilon\) and
\(a_0=\kappa\varepsilon^5\).  On the active scale
\(q=\varepsilon^4y^4\),
\[
  \varepsilon^{-5}F(\varepsilon^4y^4)
  =\kappa+y^4-y^5+\varepsilon^3y^8.
\]
For \(-256/3125<\kappa<0\), the limiting function has two simple
positive roots, while the inactive quadratic has one negative root.  For
small \(\kappa>0\), the limiting function has one positive root and the
inactive quadratic has two negative roots.  Simplicity preserves both
distributions for all sufficiently small \(\varepsilon>0\).  This example
illustrates the distinction between the coupled total bound and a complete
chamber classification.
\end{remark}

\paragraph{Upper-bound mechanism.}
By \cref{prop:shape-control}, the inactive branch satisfies
\(\sgn(c)\partial_q^n\Delta>0\).  On the active branch, the signed
derivative \(\sgn(d)\Delta_q\) is strictly increasing.
The Cross-Seam Hermite Zero Theorem therefore gives
\eqref{eq:sharp-cyclicity} with the sign \(\sgn(cd)\), uniformly in the
parameters.  The argument applies to the full \(C^2\) ramified remainder.

\subsection{Rank lifting and separated scales}

\begin{theorem}[Rank Lifting and Separated-Scale Root Realization]
\label{thm:rank-lifting-root-realization}
\label{lem:negative-root-persistence}
\label{lem:positive-root}
Assume \eqref{eq:rank-n}.  Fix \(0<r_1<\cdots<r_n\) and set
\begin{equation}
  \mathcal P_{n,t}(q)
  =c\prod_{j=1}^n(q+r_jt)
  =\sum_{j=0}^n\widehat a_j(t)q^j.
  \label{eq:sharp-polynomial}
\end{equation}
There is a lifted path \(p_t=O(t)\) satisfying
\(A(p_t)=(\widehat a_0(t),\ldots,\widehat a_{n-1}(t))\).  Along it,
the full displacement has \(n\) simple negative roots satisfying
\begin{equation}
  q_j(t)=t(-r_j+o(1)),
  \qquad
  q_n(t)<\cdots<q_1(t)<0.
  \label{eq:persistent-negative-roots}
\end{equation}
If \(cd<0\), it also has a simple positive root satisfying
\eqref{eq:positive-sharp-root}, and
\begin{equation}
  \sgn\Delta_q(q_{\mathrm{ram}}(t),p_t)=\sgn d.
  \label{eq:positive-root-derivative-sign}
\end{equation}
\end{theorem}

The contact passage produces the exponent \(\alpha_\nu=1+1/\nu\), and the
theorem treats the full remainder class of
\cref{def:ramified-displacement}.  At the smooth scale \(q=tx\), the inactive
part converges in \(C^1\) to \(c\prod_j(x+r_j)\).  At the smaller active
scale
\[
  q=t^{\gamma_{n,\nu}}y,
  \qquad \gamma_{n,\nu}=\frac{n\nu}{\nu+1}>1,
\]
the rescaled displacement converges in \(C^1_{\mathrm{loc}}((0,\infty))\)
to
\begin{equation}
  c\prod_{j=1}^nr_j+d y^{\alpha_\nu}.
  \label{eq:positive-limit-function}
\end{equation}
These two limits explain both root persistence and scale separation.  The
uniform Taylor estimates, including positive-side convergence on compact
subsets bounded away from \(y=0\), are proved in
Appendix~\ref{app:uniform-hermite}.

\begin{proof}[Proof of \cref{thm:sharp-cyclicity}]
The preceding upper-bound mechanism proves the universal estimate.  Apply
\cref{thm:rank-lifting-root-realization}: for \(cd>0\), its \(n\) negative
roots attain the bound; for \(cd<0\), the additional active root attains
the \(n+1\) bound.
\end{proof}

\subsection{Periodic-orbit stability}

\begin{corollary}[Alternating stability of the attaining orbit family]
\label[corollary]{cor:stability}
Suppose \(P(q,p)=q+\Delta(q,p)\) is an actual Poincar\'e return.  After
shrinking the germ, a simple fixed point \(q_*\) is attracting if and only
if \(\Delta_q(q_*,p)<0\).  The simple fixed points constructed in
\cref{thm:rank-lifting-root-realization}, ordered along the real axis,
alternate between attracting and repelling.
\end{corollary}

\begin{proof}
Joint continuity and \(\Delta_q(0,0)=0\) allow the germ to be chosen with
\(|\Delta_q|<1\).  Hence
\(|P'(q_*,p)|<1\) is equivalent to \(\Delta_q(q_*,p)<0\).  The rescaled
negative-root derivatives satisfy
\[
  t^{1-n}\Delta_q(q_j(t),p_t)
  \longrightarrow c\prod_{k\neq j}(r_k-r_j),
\]
so their signs alternate in spatial order; when present, the active root
has sign \(\sgn d=-\sgn c\) by
\eqref{eq:positive-root-derivative-sign}, continuing the alternation.
\end{proof}

\section{Ambient grazing strata and physical codimension}
\label{sec:ambient-codimension}

The ramified theory developed above is relative to a family with fixed
contact order.  We now embed that family in a full smooth parameter space.
Under the contact- and return-jet transversality conditions imposed below,
this yields two
complementary counts: neutral return within the fixed-contact stratum has
codimension \(n\), while maintaining contact of order \(\nu\) contributes a further
\(\nu-2\) conditions when the contact point may move in a planar phase
space.  The analysis follows a selected local contact branch.  Complete
contact loci may have self-intersections and require global projection
analysis.

Let \(M\) be a smooth surface, let \(P\subset\R^m\) be open, and let
\((X_p,h_p)\) be a smooth parameter family consisting of a vector field on
\(M\) and a switching function.  Write
\[
  H_k(z,p)=L_{X_p}^{\,k}h_p(z),
  \qquad k\geq0,
\]
and, for a fixed even integer \(\nu\geq2\), define the contact-jet map
\begin{equation}
  \mathcal J_\nu:M\times P\longrightarrow\R^\nu,
  \qquad
  \mathcal J_\nu(z,p)
  =\bigl(H_0(z,p),\ldots,H_{\nu-1}(z,p)\bigr).
  \label{eq:ambient-contact-jet}
\end{equation}
The next differential-topological theorem is independent of parity;
evenness enters only through the grazing interpretation.

\begin{definition}[Ambiently transverse contact]
\label[definition]{def:ambient-transverse-contact}
A point \((z_0,p_0)\in M\times P\) is an ambiently transverse contact of
order \(\nu\) if
\begin{equation}
  \mathcal J_\nu(z_0,p_0)=0,
  \qquad
  H_\nu(z_0,p_0)\neq0,
  \qquad
  \rank D_{(z,p)}\mathcal J_\nu(z_0,p_0)=\nu,
  \label{eq:ambient-contact-hypotheses}
\end{equation}
and \(X_{p_0}(z_0)\neq0\), \(d_zh_{p_0}(z_0)\neq0\).  Under our convention,
it is a grazing contact if, in addition,
\(-H_\nu(z_0,p_0)>0\) and the chosen side labels agree with this sign.
\end{definition}

\begin{theorem}[Contact-jet incidence and projected codimension]
\label{thm:ambient-contact-incidence}
Suppose that \((z_0,p_0)\) is an ambiently transverse contact of order
\(\nu\).  Then the following statements hold after restricting to
sufficiently small neighbourhoods of \((z_0,p_0)\) and \(p_0\).

\begin{enumerate}
\item The selected incidence branch
  \begin{equation}
    \mathcal I_\nu
    =\{(z,p):\mathcal J_\nu(z,p)=0\}
    \label{eq:ambient-incidence}
  \end{equation}
  is a smooth submanifold of \(M\times P\) of codimension \(\nu\).

\item The parameter projection
  \(\pi_P:\mathcal I_\nu\to P\) is locally an embedding.  Its image
  \begin{equation}
    \mathcal G_\nu=\pi_P(\mathcal I_\nu)
    \label{eq:ambient-grazing-stratum}
  \end{equation}
  is therefore a smooth embedded submanifold germ of \(P\), and
  \begin{equation}
    \operatorname{codim}_P\mathcal G_\nu=\nu-2.
    \label{eq:ambient-grazing-codimension}
  \end{equation}
  In particular, each \(p\in\mathcal G_\nu\) close to \(p_0\) has a unique
  order-\(\nu\) contact \(z=z(p)\) on this selected local incidence branch,
  and \(z(p)\) is smooth.
\end{enumerate}
\end{theorem}

\begin{proof}[Proof outline]
The regular-level-set theorem first gives an incidence manifold of
dimension \(m+2-\nu\).  To analyze the parameter projection, observe that
along the flow direction the contact jet has
only its last component nonzero, while a vector transverse to the seam has
its first component nonzero.  Thus the state derivative has rank two and
is injective.  A vertical tangent to the incidence manifold must therefore
vanish, so the selected parameter projection is locally an embedding.
This accounts for the subtraction of the two movable state coordinates and
establishes regularity of the selected projection.  The full rank and
localization argument is given in
\cref{app:contact-strata}.
\end{proof}

The full-rank condition in \eqref{eq:ambient-contact-hypotheses} is the
transversality of the parameter family to the order-\(\nu\) contact locus.
The rank-two calculation in \cref{app:contact-strata} is automatic
from regularity of the seam and nondegeneracy of the order-\(\nu\) contact;
the parameter directions must supply the remaining \(\nu-2\) ranks.  This
explains geometrically why the existence of an unmarked local quadratic
contact is an open condition, while every increase of the contact order by
one imposes
one further ambient condition.  Prescribing a recurrent orbit or fixing
the contact point can impose additional conditions, as recorded below.

\subsection{Intrinsic character of the contact stratum}

We next verify that \(\mathcal G_\nu\) and its codimension are invariant
under changes of the auxiliary defining function, the speed used to
parametrize the reference trajectories, and phase-space coordinates.  The
key algebraic fact is the triangularity of iterated Lie derivatives.

\begin{lemma}[Triangular invariance of the contact jet]
\label[lemma]{lem:contact-jet-triangular-invariance}
Let \(u(z,p)\neq0\) and \(a(z,p)>0\) be smooth, and set
\[
  \widehat h_p=u_p h_p,
  \qquad
  \widehat X_p=a_pX_p.
\]
There is a smooth lower-triangular matrix
\(\mathsf T_\nu(z,p)\in\operatorname{GL}(\nu,\R)\) such that
\begin{equation}
  \widehat{\mathcal J}_\nu=\mathsf T_\nu\mathcal J_\nu,
  \qquad
  \operatorname{diag}\mathsf T_\nu
  =\bigl(u,ua,ua^2,\ldots,ua^{\nu-1}\bigr).
  \label{eq:contact-jet-triangular-law}
\end{equation}
At every zero of \(\mathcal J_\nu\),
\begin{equation}
  D\widehat{\mathcal J}_\nu=\mathsf T_\nu D\mathcal J_\nu,
  \qquad
  L_{\widehat X}^{\,\nu}\widehat h
  =ua^\nu L_X^{\,\nu}h.
  \label{eq:contact-jet-rank-law}
\end{equation}
Thus the incidence equations, their rank, and the order of contact are
unchanged.  If \(u>0\), the labels of the two switching sides are unchanged
as well.

If \(y=\Phi_p(z)\) is a parameter-dependent phase-space diffeomorphism,
\(\widehat X_p=(\Phi_p)_*X_p\), and
\(\widehat h_p=h_p\circ\Phi_p^{-1}\), then
\begin{equation}
  \widehat H_k(\Phi_p(z),p)=H_k(z,p)
  \qquad(k\geq0).
  \label{eq:contact-jet-coordinate-covariance}
\end{equation}
Consequently, phase-space and parameter coordinate changes carry the
incidence germ diffeomorphically to the corresponding incidence germ and
preserve its projected codimension.
\end{lemma}

\begin{proof}[Proof outline]
Leibniz' rule gives the lower-triangular jet law, and naturality gives the
coordinate covariance.  The rank and projected-incidence calculations are
in \cref{app:contact-strata}.
\end{proof}

\begin{remark}[An intrinsic stratum with auxiliary parametrizations]
The lemma shows that the germ of the set of systems possessing the selected
contact is intrinsic.  The individual entries of \(\mathcal J_\nu\) and
the choice of its \(\nu-2\) transverse parameters depend on the auxiliary
representation.  An invertible triangular transformation relates these
choices along the incidence set and preserves the transversality and
codimension data.
\end{remark}

\subsection{Persistence under perturbation}

The stratum persists under perturbations of the whole parameter family:
each nearby perturbed family contains a nearby codimension-\(\nu-2\)
contact stratum.  Individual parameters outside that stratum generally
lose the high-order contact.

\begin{proposition}[Persistence of the selected ambient stratum]
\label[proposition]{prop:ambient-stratum-persistence}
Under the hypotheses of \cref{thm:ambient-contact-incidence}, every
sufficiently small smooth perturbation \((\widetilde X_p,\widetilde h_p)\)
in the joint \(C^\nu\) topology in \((z,p)\), on a fixed compact
neighbourhood, has a selected nearby order-\(\nu\) incidence germ
\(\widetilde{\mathcal I}_\nu\) in a fixed coordinate neighbourhood of
\((z_0,p_0)\).  After shrinking that neighbourhood uniformly, its parameter
projection is a local embedding and
\[
  \operatorname{codim}_P\widetilde{\mathcal G}_\nu=\nu-2.
\]
For a smooth finite-dimensional perturbation family, the selected graph and
projected embedding vary smoothly.  The grazing sign persists when
\(-H_\nu(z_0,p_0)>0\).
\end{proposition}

\begin{proof}[Proof outline]
Apply the implicit-function and constant-rank theorems to a fixed nonzero
\(\nu\)-minor.  Uniform neighbourhoods, the projected embedding, and smooth
dependence are established in \cref{app:contact-strata}.
\end{proof}

\subsection{Neutrality within the ambient grazing stratum}

We now combine the ambient contact count with the smooth coefficient map of
\eqref{eq:coefficient-map}.  The local contact jet supplies the grazing
conditions, while the coefficient map carries global return information.
We therefore impose transversality of the latter as an independent
hypothesis.

\begin{theorem}[Total physical codimension]
\label{thm:total-physical-codimension}
Let \((z_0,p_0)\) satisfy the hypotheses of
\cref{thm:ambient-contact-incidence}, and let \(\mathcal G_\nu\) be the
selected projected contact stratum.  Suppose that, for \(p\in\mathcal
G_\nu\), a smooth choice of local sections and a seam-centred coordinate
produces a smooth reference displacement \(S(q,p)\).  Define
\begin{equation}
  A_\nu:\mathcal G_\nu\longrightarrow\R^n,
  \qquad
  A_\nu(p)=
  \left(
    S(0,p),\frac{\partial_qS(0,p)}{1!},\ldots,
    \frac{\partial_q^{n-1}S(0,p)}{(n-1)!}
  \right).
  \label{eq:ambient-neutral-coefficient-map}
\end{equation}
Assume
\begin{equation}
  A_\nu(p_0)=0,
  \qquad
  c_n:=\frac{\partial_q^nS(0,p_0)}{n!}\neq0,
  \qquad
  \rank d(A_\nu)_{p_0}=n,
  \label{eq:ambient-neutral-transversality}
\end{equation}
where the differential is taken on \(T_{p_0}\mathcal G_\nu\).  Then the
order-\(n\) neutral grazing locus
\begin{equation}
  \mathcal N_{n,\nu}
  =\{p\in\mathcal G_\nu:A_\nu(p)=0\}
  \label{eq:ambient-neutral-locus}
\end{equation}
is a smooth embedded submanifold germ and
\begin{equation}
  \operatorname{codim}_P\mathcal N_{n,\nu}=n+\nu-2.
  \label{eq:total-physical-codimension}
\end{equation}
In particular, the hypotheses force \(m\geq n+\nu-2\).
\end{theorem}

\begin{proof}
By \cref{thm:ambient-contact-incidence}, \(\mathcal G_\nu\) is an embedded
submanifold germ of codimension \(\nu-2\).  The rank condition in
\eqref{eq:ambient-neutral-transversality} says that \(A_\nu\) is transverse
to \(0\in\R^n\).  Hence the regular-level-set theorem applied on
\(\mathcal G_\nu\) makes \(\mathcal N_{n,\nu}\) a codimension-\(n\)
submanifold of \(\mathcal G_\nu\).  Codimensions add along embedded
submanifolds, which gives \eqref{eq:total-physical-codimension}.
\end{proof}

\begin{corollary}[Relative sharp cyclicity on the ambient stratum]
\label[corollary]{cor:ambient-relative-cyclicity}
Under the hypotheses of \cref{thm:total-physical-codimension}, suppose the
physical displacement on \(\mathcal G_\nu\) has the full-remainder form of
\cref{def:ramified-displacement}, with \(cd\neq0\).  Its cyclicity under
parameter paths in \(\mathcal G_\nu\) is
\begin{equation}
  \Cyc_{\mathcal G_\nu}=
  \begin{cases}
    n,&cd>0,\\
    n+1,&cd<0.
  \end{cases}
  \label{eq:ambient-relative-sharp-cyclicity}
\end{equation}
The attaining paths lie in \(\mathcal G_\nu\).  For actual Poincar\'e
displacements on a common return domain, their simple roots are hyperbolic
periodic orbits after shrinking the germ.
\end{corollary}

\begin{proof}
On \(\mathcal G_\nu\), condition
\eqref{eq:ambient-neutral-transversality} is the rank hypothesis of
\cref{thm:sharp-cyclicity}.  Apply that theorem and
\cref{cor:stability}.
\end{proof}

\begin{remark}[Intrinsic relative and ambient counts]
\label[remark]{rem:relative-ambient-codimension}
An increasing seam-centred change of section coordinate preserves the
order-\(n\) zero of the smooth reference displacement.  At a point where
the first \(n\) coefficients vanish, the change induces a lower-triangular
action with nonzero diagonal on their parameter differentials;
multiplication by a smooth nonvanishing displacement unit has the same
property.  Thus the
rank condition \eqref{eq:ambient-neutral-transversality} is independent of
these choices.  Indeed, if \(q=\psi(\widetilde q,p)\) is an increasing
seam-centred section coordinate, the actual return displacement changes by
the exact conjugacy identity
\[
  \widetilde\Delta(\widetilde q,p)
  =\psi_p^{-1}\!\left(
      \psi_p(\widetilde q)
      +\Delta(\psi_p(\widetilde q),p)
    \right)-\widetilde q.
\]
Taking the first \(n\) jets at a central displacement whose lower jets
vanish gives precisely the invertible triangular action just described;
the exact conjugacy identity therefore establishes the required invariance.
On the fixed-contact stratum, the coefficient rank gives the intrinsic
relative count, while allowing the planar contact to move contributes the
ambient conditions:
\[
  \operatorname{codim}_{\mathcal G_\nu}\mathcal N_{n,\nu}=n,
  \qquad
  \underbrace{n}_{\text{neutral return}}+
  \underbrace{(\nu-2)}_{\text{moving contact}}=n+\nu-2.
\]
The subtraction of two uses both state coordinates of the unmarked contact
point.  Fixing that point instead gives the generic parameter count \(\nu\);
requiring it to lie on another prescribed orbit or invariant set may add
global incidence conditions.  Thus relative cyclicity and relative neutral
codimension are independent of \(\nu\), whereas ambient physical codimension
grows with the grazing order.  For \(\nu=2\), the selected contact stratum is
open and the two counts coincide.
\end{remark}

\section{The universal positive-part cap and weighted crossover}
\label{sec:transverse-cap}

The fixed-contact passage in \cref{thm:localized-active-passage} describes
the central order-\(\nu\) stratum.  Transverse contact jets replace its
truncated monomial by the positive-part area of a coercive polynomial.  The
result below extends the passage theorem to the corresponding weighted
blow-up.  It provides a uniform \emph{value estimate} when active components
are born or merge.  The full smooth ramified remainder and derivative
control require the additional structure developed in
\cref{sec:transverse-cyclicity}.

\paragraph{Notation for the weighted passage.}
The physical signed section coordinate is \(q=I_{\rm in}\).  Under the
blow-up, \(q=r^\nu\rho\) and
\((\mathbf u_r)_j=r^{\nu-j}u_j\) for \(1\leq j\leq\nu-2\): \(r\downarrow0\)
is the scale, \((\rho,\mathbf u)\) remains in a bounded normalized set, and
\(\mathbf u_r\) contains the physical centered contact coefficients.  We
retain \(p\) for ambient parameters.  Later, \(\eta\) denotes a normalized
linear tilt whose physical coefficient is \(r^{\nu-1}b\eta\).

Throughout, \(\nu\geq2\) is even and \(a>0\).  For
\(\mathbf u=(u_1,\ldots,u_{\nu-2})\), with the vector absent when
\(\nu=2\), define
\begin{equation}
  \mathcal C_{\nu,a}(\rho,\mathbf u)
  =\int_{\R}\left(\rho-a t^\nu-
    \sum_{j=1}^{\nu-2}u_jt^j\right)_+\dd t.
  \label{eq:transverse-cap-functional}
\end{equation}
Its natural dilation is
\begin{equation}
  \delta_r(\rho,u_1,\ldots,u_{\nu-2})
  =\left(r^\nu\rho,r^{\nu-1}u_1,\ldots,r^2u_{\nu-2}\right).
  \label{eq:cap-weighted-dilation}
\end{equation}
The constant coefficient is the penetration \(\rho\); translation of the
flow coordinate removes the \(t^{\nu-1}\) coefficient.  The remaining
\(\nu-2\) coefficients are precisely the planar transverse contact-jet
directions.

\subsection{Preparation and cap geometry}

\begin{lemma}[Transverse cap preparation]
\label[lemma]{lem:transverse-cap-preparation}
Let \(X_p^-\) be a smooth family of nonvanishing vector fields and \(h_p\)
a smooth family of switching functions near an order-\(\nu\) contact.
Choose an \(X_p^-\)-flow box \((t,I)\), so \(X_p^-=\partial_t\), and
orient \(I\) so the seam is
\[
  I=\Phi(t,p),
  \qquad \Phi(t,0)=at^\nu+O(t^{\nu+1}),
  \qquad a>0.
\]
There are smooth parameter-dependent translations of \(t\) and \(I\)
for which the constant and \(t^{\nu-1}\) coefficients vanish.  In these
coordinates put
\[
  u_j(p)=\frac1{j!}\partial_t^j\Phi(0,p),
  \qquad1\leq j\leq\nu-2.
\]
If the centered coefficient map has rank \(\nu-2\), then on a transverse
parameter slice these coefficients are coordinates and
\begin{equation}
  \Phi_{\mathbf u}(t)
  =a(\mathbf u)t^\nu+
    \sum_{j=1}^{\nu-2}u_jt^j+t^{\nu+1}R(t,\mathbf u),
  \qquad a(0)=a,
  \label{eq:prepared-cap-graph}
\end{equation}
with \(R\) smooth.  Moreover,
\begin{equation}
  h_{\mathbf u}(t,I)
  =g(t,I,\mathbf u)(I-\Phi_{\mathbf u}(t)),
  \qquad g>0.
  \label{eq:prepared-cap-switching-function}
\end{equation}
After shrinking the common box and parameter slice, there is
\(\kappa_*>0\) such that
\begin{equation}
  \partial_t^\nu\Phi_{\mathbf u}(t)\geq\kappa_*
  \qquad (|t|\leq T).
  \label{eq:prepared-cap-finite-type}
\end{equation}
\end{lemma}

The preparation uses the implicit-function theorem, centering at the
unique zero of \(\partial_t^{\nu-1}\Phi\), the constant-rank theorem, and
Hadamard division.  Details are in Appendix~\ref{app:passage}.

\begin{lemma}[Weighted homogeneity and discriminant continuity]
\label[lemma]{lem:cap-weighted-homogeneity}
The cap in \eqref{eq:transverse-cap-functional} is continuous, including
where its positive set changes topology, and
\begin{equation}
  \mathcal C_{\nu,a}
  (r^\nu\rho,r^{\nu-1}u_1,\ldots,r^2u_{\nu-2})
  =r^{\nu+1}\mathcal C_{\nu,a}(\rho,\mathbf u).
  \label{eq:cap-weighted-homogeneity}
\end{equation}
\end{lemma}

The substitution \(t=rs\) proves the identity; coercivity gives a common
compact support on compact coefficient sets, and
\(|x_+-y_+|\leq|x-y|\) gives continuity.  The complete compact-uniform
argument is recorded in Appendix~\ref{app:passage}.

\subsection{Localized active-region passage in transverse weights}

After shrinking once, fix
\(\mathcal Q=[-T,T]\times(-I_*,I_*)\) and suppose
\[
  X^-_{\mathbf u}=\partial_t,
  \qquad
  X^+_{\mathbf u}-X^-_{\mathbf u}=h_{\mathbf u}W_{\mathbf u},
  \qquad
  W_{\mathbf u}=A_{\mathbf u}\partial_t+B_{\mathbf u}\partial_I.
\]
The physical field is
\begin{equation}
  X_{\mathbf u}=\partial_t+(h_{\mathbf u})_+W_{\mathbf u},
  \label{eq:physical-transverse-cap-field}
\end{equation}
and the leading mismatch is
\begin{equation}
  \beta=dh_0(0,0)[W_0(0,0)]=g(0,0,0)B_0(0,0).
  \label{eq:transverse-cap-beta}
\end{equation}
Launch the physical trajectory at \((-T,q)\), and let
\(\mathcal D_{\mathrm{cap}}(q,\mathbf u)\) be its outgoing
\(I\)-coordinate minus that of the reference \(X^-\)-trajectory.

\begin{theorem}[Localized Active-Region Passage / Weighted Blow-up:
transverse value estimate]
\label{thm:transverse-cap-blow-up}
Assume \eqref{eq:prepared-cap-graph}--%
\eqref{eq:physical-transverse-cap-field}.  For every compact
\(K\subset\R\times\R^{\nu-2}\), there are \(C_K,r_K>0\) such that for
\((\rho,\mathbf u)\in K\) and \(0<r<r_K\), the physical trajectory stays
in \(\mathcal Q\), reaches \(t=T\), and
\begin{equation}
  \left|
  \mathcal D_{\mathrm{cap}}
  (r^\nu\rho,r^{\nu-1}u_1,\ldots,r^2u_{\nu-2})
  -\beta r^{\nu+1}\mathcal C_{\nu,a}(\rho,\mathbf u)
  \right|
  \leq C_Kr^{\nu+2}.
  \label{eq:physical-cap-weighted-limit}
\end{equation}
The estimate is uniform across every polynomial discriminant stratum in
\(K\), including births, mergers, and disappearances of positive
components.
\end{theorem}

\paragraph{Proof outline and the uniform step.}
Coercivity localizes the active set to \(|t|\leq M_Kr\), and continuation
gives \(I(t)=r^\nu\rho+O_K(r^{\nu+1})\).  In the scaled coordinate \(t=rs\),
the positive-part Lipschitz estimate compares the rescaled physical
integrand uniformly with the model cap integrand across the discriminant.
Appendix~\ref{app:passage} contains the localization and bootstrap proof.

\begin{corollary}[Regular transmission of the transverse cap]
\label[corollary]{cor:transverse-cap-transmission}
Let \(\mathcal R_{\mathbf u}\) be a smooth family of regular outgoing maps,
with \(\mathcal R_0(0)=0\) and \(L_0=\mathcal R_0'(0)\neq0\), and define
\[
  P_{\mathbf u}(q)=\mathcal R_{\mathbf u}
  (q+\mathcal D_{\mathrm{cap}}(q,\mathbf u)),
  \qquad P^-_{\mathbf u}(q)=\mathcal R_{\mathbf u}(q).
\]
Then, uniformly on every compact normalized set,
\begin{equation}
  P_{\mathbf u_r}(r^\nu\rho)-P^-_{\mathbf u_r}(r^\nu\rho)
  =L_0\beta r^{\nu+1}\mathcal C_{\nu,a}(\rho,\mathbf u)
  +O_K(r^{\nu+2}),
  \label{eq:transmitted-physical-cap-weighted-limit}
\end{equation}
where \(\mathbf u_r=(r^{\nu-1}u_1,\ldots,r^2u_{\nu-2})\).  Compatible
finite-time transitions on a common itinerary give the same estimate for
the actual first return provided the sections remain transverse and the
composed orbit has no earlier intersection with the incoming section, as in
the recognition hypotheses of \cref{prop:reference-return-composition}.
Root bounds additionally require derivative control.
\end{corollary}

\begin{proof}
Apply \cref{prop:weighted-value-transmission} to
\cref{thm:transverse-cap-blow-up}.
\end{proof}

\begin{corollary}[Central fixed-order coefficient]
\label[corollary]{cor:central-cap-coefficient}
On the central stratum \(\mathbf u=0\),
\begin{equation}
  \mathcal C_{\nu,a}(\rho,0)
  =\frac{2\nu}{\nu+1}a^{-1/\nu}\rho_+^{1+1/\nu}.
  \label{eq:central-cap-evaluation}
\end{equation}
Consequently,
\begin{equation}
  \mathcal D_{\mathrm{cap}}(q,0)
  =\beta\frac{2\nu}{\nu+1}a^{-1/\nu}q^{1+1/\nu}
  +O(q^{1+2/\nu})
  \quad(q\downarrow0).
  \label{eq:central-cap-physical-asymptotic}
\end{equation}
With \(a=\kappa_\nu/\nu!\), this is
\eqref{eq:grazing-leading-coefficient}.  On this central slice,
\cref{thm:localized-active-passage} strengthens the displayed value
asymptotic to the full smooth ramified expansion.
\end{corollary}

\begin{proof}
Substitute \(t=(\rho/a)^{1/\nu}s\) in the cap, then use
\cref{thm:transverse-cap-blow-up} with normalized point \((1,0)\).
\end{proof}

\subsection{A one-parameter tilt and the quadratic crossover}

For \(\nu\geq4\), restrict the contact-jet parameters to the linear
coefficient.  Let
\(b\neq0\) and put
\begin{equation}
  V_\eta(t)=a t^\nu+b\eta t,
  \qquad
  \mathcal C_\nu^{\mathrm{tilt}}(\rho,\eta)
  =\int_\R(\rho-V_\eta(t))_+\dd t.
  \label{eq:tilted-cap}
\end{equation}

\begin{figure}[tbp]
  \centering
  \includegraphics[width=0.96\textwidth]{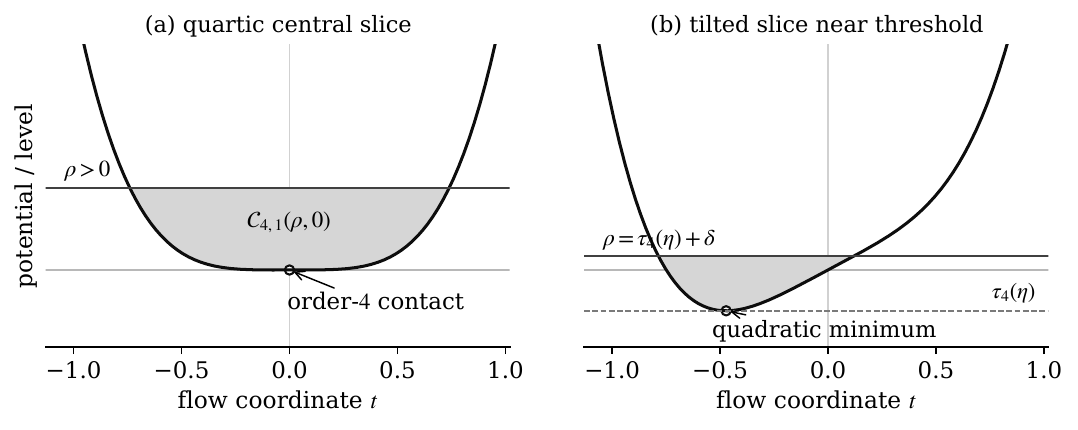}
  \caption{Two potential--level geometries underlying the weighted
  crossover for \(V_\eta(t)=t^4+\eta t\); the shaded areas represent the
  cap.  \Cref{prop:tilted-cap-crossover} supplies the exact weighted
  dilation and the \(5/4\) and \(3/2\) onset laws.  The displayed parameter
  values, \(\rho=0.30\) on the central slice and \(\eta=0.42\) with
  \(\rho=\tau_4(0.42)+0.20\) on the tilted slice, show the corresponding
  local shapes.}
  \label{fig:cap-crossover}
\end{figure}

\begin{proposition}[Threshold and quadratic crossover of the tilted cap]
\label[proposition]{prop:tilted-cap-crossover}
For every \(\eta\neq0\), \(V_\eta\) has a unique global minimum at
\begin{equation}
  t_\eta=-\sgn(b\eta)
  \left(\frac{|b\eta|}{a\nu}\right)^{1/(\nu-1)}.
  \label{eq:tilted-cap-minimizer}
\end{equation}
Its threshold and curvature are
\begin{align}
  \tau_\nu(\eta)
  &=-(\nu-1)a
  \left(\frac{|b\eta|}{a\nu}\right)^{\nu/(\nu-1)},
  \label{eq:tilted-cap-threshold}\\
  \kappa_{\mathrm{eff}}(\eta)
  &=a\nu(\nu-1)
  \left(\frac{|b\eta|}{a\nu}\right)^{(\nu-2)/(\nu-1)}.
  \label{eq:tilted-cap-effective-curvature}
\end{align}
The cap vanishes for \(\rho\leq\tau_\nu(\eta)\).  For each fixed
\(\eta\neq0\), as \(\delta=\rho-\tau_\nu(\eta)\downarrow0\),
\begin{equation}
  \mathcal C_\nu^{\mathrm{tilt}}
  (\tau_\nu(\eta)+\delta,\eta)
  =\frac{4\sqrt2}{3\sqrt{\kappa_{\mathrm{eff}}(\eta)}}
  \delta^{3/2}+O_\eta(\delta^2).
  \label{eq:tilted-cap-three-halves}
\end{equation}
Moreover,
\begin{equation}
  \mathcal C_\nu^{\mathrm{tilt}}(r^\nu\rho,r^{\nu-1}\eta)
  =r^{\nu+1}\mathcal C_\nu^{\mathrm{tilt}}(\rho,\eta).
  \label{eq:tilted-cap-crossover-scaling}
\end{equation}
Thus the crossover scales are
\begin{equation}
  |t|\asymp|\eta|^{1/(\nu-1)},
  \qquad |\rho|\asymp|\eta|^{\nu/(\nu-1)}.
  \label{eq:tilted-cap-crossover-scales}
\end{equation}
\end{proposition}

\begin{proof}
Since \(V_\eta'(t)=a\nu t^{\nu-1}+b\eta\) is strictly increasing, it has
the unique zero \eqref{eq:tilted-cap-minimizer}.  The identity
\(b\eta=-a\nu t_\eta^{\nu-1}\) gives the threshold, and direct
differentiation gives the curvature.  Taylor expansion at \(t_\eta\)
gives
\[
  V_\eta(t_\eta+x)-\tau_\nu(\eta)
  =\tfrac12\kappa_{\mathrm{eff}}(\eta)x^2+O_\eta(|x|^3).
\]
The active interval has \(|x|=O_\eta(\sqrt\delta)\).  Substitution
\(x=\sqrt\delta\,y\) and the positive-part Lipschitz inequality give
\[
  \delta^{3/2}\int_\R
  \left(1-\tfrac12\kappa_{\mathrm{eff}}(\eta)y^2\right)_+\dd y
  +O_\eta(\delta^2),
\]
whose integral is the coefficient in
\eqref{eq:tilted-cap-three-halves}.  Weighted homogeneity gives
\eqref{eq:tilted-cap-crossover-scaling}; the stated scales follow from the
minimum formulas.
\end{proof}

\begin{corollary}[Physical meaning of the crossover]
\label[corollary]{cor:physical-tilted-cap-crossover}
On the slice \(\mathbf u=(b\eta,0,\ldots,0)\), uniformly on compact
normalized \((\rho,\eta)\)-sets,
\begin{equation}
  \mathcal D_{\mathrm{cap}}
  (r^\nu\rho,r^{\nu-1}b\eta,0,\ldots,0)
  =\beta r^{\nu+1}\mathcal C_\nu^{\mathrm{tilt}}(\rho,\eta)
  +O_K(r^{\nu+2}).
  \label{eq:physical-tilted-cap-weighted-limit}
\end{equation}
For the physical prepared graph, write
\[
  \Phi_\eta(t)=\Phi_{(b\eta,0,\ldots,0)}(t).
\]
For every fixed sufficiently small \(\eta\neq0\), this graph has a selected
nondegenerate minimum \(\widehat t_\eta\) near the model minimum
\(t_\eta\).  Put
\[
  \widehat\tau_\eta=\Phi_\eta(\widehat t_\eta),
  \qquad
  \widehat\kappa_\eta=\Phi_\eta''(\widehat t_\eta)>0,
  \qquad
  \widehat\beta_\eta
  =dh_\eta(\widehat t_\eta,\widehat\tau_\eta)[W_\eta].
\]
Then the prepared tilted seam has ordinary quadratic grazing there and
\begin{equation}
  \mathcal D_{\mathrm{cap}}
  (\widehat\tau_\eta+\delta,(b\eta,0,\ldots,0))
  =\frac{4\sqrt2\,\widehat\beta_\eta}
  {3\sqrt{\widehat\kappa_\eta}}\delta^{3/2}
  +o_\eta(\delta^{3/2}).
  \label{eq:physical-tilted-cap-three-halves}
\end{equation}
As \(\eta\to0\) with its sign fixed,
\[
  \frac{\widehat t_\eta}{t_\eta}\longrightarrow1,
  \qquad
  \frac{\widehat\kappa_\eta}{\kappa_{\mathrm{eff}}(\eta)}
  \longrightarrow1,
  \qquad
  \frac{\widehat\tau_\eta}{\tau_\nu(\eta)}\longrightarrow1.
\]
This fixed-tilt asymptotic is pointwise; its constants may degenerate as
\(\eta\to0\).  Equation \eqref{eq:physical-tilted-cap-weighted-limit} is
the uniform joint crossover statement.
\end{corollary}

\begin{proof}
The weighted statement is the transverse passage theorem restricted to the
tilt slice.  For the physical graph, scale
\(t=|\eta|^{1/(\nu-1)}u\) in \(\Phi_\eta'(t)=0\).  The prepared remainder
is lower order in the scaled equation, while the model critical point is
simple.  The implicit-function theorem gives \(\widehat t_\eta\) and the
three displayed comparisons.  Translate to that nondegenerate minimum and
apply the quadratic case of \cref{thm:localized-active-passage} in the
smoothly rescaled transverse coordinate.  This gives
\eqref{eq:physical-tilted-cap-three-halves} in the original \(I\)-coordinate.
Continuity keeps \(\widehat\beta_\eta\neq0\) when the central mismatch is
nonzero.
\end{proof}

The value estimate remains uniform for general multiwell cap geometry.
Derivative control follows from the finite physical event discriminant
provided by the prepared finite-type bound
\eqref{eq:prepared-cap-finite-type}.  The multi-event argument in
\cref{thm:prepared-multiwell-boundary-curvature,%
thm:physical-multiwell-cyclicity} supplies derivative
monotonicity and the sharp physical root bound on the complete prepared
contact-jet unfolding.

\section{Exact and derivative-controlled transverse cyclicity}
\label{sec:transverse-cyclicity}

The weighted passage theorem gives a uniform value estimate across the
contact-jet unfolding.  A root count also requires derivative control.  We
first prove active-side convexity for the exact polynomial cap, then give a
value-only counterexample, and finally recover derivative monotonicity from
all moving entry and exit events of the prepared physical passage.

\subsection{Sublevel geometry of the exact cap}

Retain \cref{eq:transverse-cap-functional} and write
\begin{equation}
  V_{\mathbf u}(t)=at^\nu+\sum_{j=1}^{\nu-2}u_jt^j,
  \qquad
  \tau(\mathbf u)=\min_{t\in\R}V_{\mathbf u}(t),
  \label{eq:cap-potential-threshold}
\end{equation}
and
\begin{equation}
  \mathcal L(q,\mathbf u)
  =\left|\{t\in\R:V_{\mathbf u}(t)<q\}\right|.
  \label{eq:cap-sublevel-length}
\end{equation}

\begin{lemma}[Positive-part cap and quantitative sublevel growth]
\label[lemma]{lem:cap-strong-convexity}
The map \(q\mapsto\mathcal C_{\nu,a}(q,\mathbf u)\) is \(C^1\), vanishes
together with its first derivative for \(q\leq\tau(\mathbf u)\), and
\begin{equation}
  \partial_q\mathcal C_{\nu,a}(q,\mathbf u)
  =\mathcal L(q,\mathbf u).
  \label{eq:cap-first-derivative-length}
\end{equation}
There are \(C>0\) and a neighbourhood of the origin such that, whenever
\(q_2>q_1>\tau(\mathbf u)\),
\begin{equation}
  \mathcal L(q_2,\mathbf u)-\mathcal L(q_1,\mathbf u)
  \geq C^{-1}r^{1-\nu}(q_2-q_1),
  \label{eq:cap-strong-convexity}
\end{equation}
where
\begin{equation}
  r=\max\left\{|q_1|^{1/\nu},|q_2|^{1/\nu},
  |u_j|^{1/(\nu-j)}:1\leq j\leq\nu-2\right\}.
  \label{eq:cap-pair-gauge}
\end{equation}
The estimate holds through critical values and its modulus diverges
uniformly toward the central contact.
\end{lemma}

\begin{proof}
Dominated convergence gives \eqref{eq:cap-first-derivative-length}, and the
sublevel length tends to zero at the global minimum.  For
\(q\in[q_1,q_2]\), put \(t=rs\).  The normalized coefficients lie in a
fixed compact set, so coercivity places every real level root in
\(|s|\leq M\).  Hence
\begin{equation}
  |V_{\mathbf u}'(t)|\leq C_0r^{\nu-1}
  \label{eq:cap-level-slope-bound}
\end{equation}
at every relevant root.  For a regular value, the one-dimensional coarea
formula \cite[Theorem~3.2.22]{Federer1969} gives
\begin{equation}
  \partial_q\mathcal L(q,\mathbf u)
  =\sum_{V_{\mathbf u}(t)=q}\frac1{|V_{\mathbf u}'(t)|}.
  \label{eq:cap-coarea-curvature}
\end{equation}
There are at least two outer level roots.  Thus the sum is bounded below by
\(C_1^{-1}r^{1-\nu}\).  Decompose into the finitely many monotonicity
intervals of \(V_{\mathbf u}\), change variables on each, and use monotone
convergence.  Critical values form a null set, and internal wells only add
terms, proving \eqref{eq:cap-strong-convexity} even across births and
mergers.
\end{proof}

\begin{theorem}[Sharp cyclicity for the exact full transverse cap]
\label{thm:exact-transverse-cap-cyclicity}
Let \(n\geq2\), let \(\nu\geq4\) be even, and suppose
\begin{equation}
  F(q,p,\mathbf u)
  =S(q,p,\mathbf u)+d(p,\mathbf u)\mathcal C_{\nu,a}(q,\mathbf u),
  \label{eq:exact-transverse-cap-family}
\end{equation}
where \(S\) is jointly \(C^{n+1}\), \(d\) is jointly \(C^2\), and
\begin{equation}
  \begin{aligned}
    \partial_q^jS(0,0,0)&=0 &&(0\leq j<n),\\
    c&=\frac{\partial_q^nS(0,0,0)}{n!}\neq0,
    &d_0&=d(0,0)\neq0.
  \end{aligned}
  \label{eq:exact-cap-central-data}
\end{equation}
Then
\begin{equation}
  \Cyc_{(0,0,0)}(F)
  \leq
  \begin{cases}
    n,&cd_0>0,\\
    n+1,&cd_0<0.
  \end{cases}
  \label{eq:exact-transverse-cap-bound}
\end{equation}
If the central coefficient map
\begin{equation}
  p\longmapsto\left(S(0,p,0),\partial_qS(0,p,0),\ldots,
  \frac{\partial_q^{n-1}S(0,p,0)}{(n-1)!}\right)
  \label{eq:exact-cap-rank-map}
\end{equation}
has rank \(n\), the applicable bound is attained on \(\mathbf u=0\).
The upper bound is uniform across the complete exact cap unfolding.
\end{theorem}

\begin{proof}
On \(q\leq\tau(\mathbf u)\), the cap and its first derivative vanish, so
\(F=S\) and \(\sgn(c)\partial_q^nF>0\) after shrinking.  On the active
side, \(S_q\) has a uniform Lipschitz constant \(M\), while \(d\) has the
sign of \(d_0\) and modulus at least \(|d_0|/2\).  Thus for
\(q_2>q_1>\tau\), \cref{lem:cap-strong-convexity} gives
\[
  \sgn(d_0)(F_q(q_2)-F_q(q_1))
  \geq\left(\frac{|d_0|}{2C}r^{1-\nu}-M\right)(q_2-q_1)>0
\]
on a sufficiently small neighbourhood.  Since \(F\) is \(C^1\) across
the moving threshold, \cref{cor:moving-threshold-cyclicity} gives
\eqref{eq:exact-transverse-cap-bound} and isolation.  On \(\mathbf u=0\),
\cref{cor:central-cap-coefficient} reduces the cap to a nonzero multiple of
\(q_+^{1+1/\nu}\); \cref{thm:sharp-cyclicity} and the rank hypothesis give
attainment.
\end{proof}

\subsection{The limitation of value-only control}

The following explicit example identifies the boundary of the weighted
passage theorem.  It has exactly the next weighted value order, while the
derivative structure required by the moving-threshold corollary is absent.

\begin{proposition}[Unbounded cyclicity under value-only control]
\label[proposition]{prop:value-only-no-cyclicity}
Fix even \(\nu\geq4\), \(n\geq2\), and choose
\begin{equation}
  \frac1\nu<\gamma<\frac2\nu,
  \qquad
  R(q)=
  \begin{cases}
    q^{1+2/\nu}\sin(q^{-\gamma}),&q>0,\\
    0,&q\leq0.
  \end{cases}
  \label{eq:oscillatory-weighted-remainder}
\end{equation}
Then \(R\in C^1(\R)\), \(R(0)=R'(0)=0\), and on every compact
\(q\)-set
\begin{equation}
  |R(r^\nu q)|\leq C_Kr^{\nu+2}.
  \label{eq:oscillatory-weighted-bound}
\end{equation}
Nevertheless, for every nonzero \(c,d\), the family
\begin{equation}
  a_0+a_1q+cq^n
  +d\frac{2\nu}{\nu+1}a^{-1/\nu}q_+^{1+1/\nu}+R(q)
  \label{eq:value-only-counterexample-family}
\end{equation}
has unbounded local cyclicity at the origin.
\end{proposition}

\begin{proof}
For \(q>0\),
\[
  R'(q)=\left(1+\frac2\nu\right)q^{2/\nu}\sin(q^{-\gamma})
  -\gamma q^{2/\nu-\gamma}\cos(q^{-\gamma}).
\]
The choice \(\gamma<2/\nu\) proves the \(C^1\) extension, and the weighted
bound follows directly from
\((r^\nu q)^{1+2/\nu}=r^{\nu+2}q^{1+2/\nu}\).

For \(q>0\), let
\[
  H(q)=cq^n+d\frac{2\nu}{\nu+1}a^{-1/\nu}q^{1+1/\nu}
\]
and choose \(q_k\downarrow0\) with \(q_k^{-\gamma}=2\pi M_k\),
\(M_k\in\mathbb N\).  Put
\begin{equation}
  a_{1,k}=-H'(q_k),
  \qquad a_{0,k}=q_kH'(q_k)-H(q_k).
  \label{eq:oscillatory-affine-path}
\end{equation}
Both parameters tend to zero.  Given an integer \(N\geq1\), define the first
\(N+1\) neighboring phase extrema, for \(0\leq j\leq N\), by
\[
  q_{k,j}=\left(2\pi M_k+(j+\tfrac12)\pi\right)^{-1/\gamma},
\]
At these points the oscillatory term alternates in sign and has modulus
comparable with
\(q_k^{1+2/\nu}\).  Taylor's theorem gives
\[
  |H(q_{k,j})-H(q_k)-H'(q_k)(q_{k,j}-q_k)|
  \leq CN^2q_k^{1+1/\nu+2\gamma}
  =o(q_k^{1+2/\nu}),
\]
because \(2\gamma>1/\nu\).  Choose \(k\) sufficiently large that all points
and both parameters lie in the prescribed
neighbourhoods.  The family values alternate, producing roots in \(N\)
disjoint intervals.  Real analyticity for \(q>0\) excludes a zero interval.
\end{proof}

Thus finite cyclicity requires derivative information beyond the value
estimate in \cref{thm:transverse-cap-blow-up}, even with its sharp
next-weight error.
This example belongs to the larger abstract value-controlled class.  In the
ramified coordinate \(s=q^{1/\nu}\), its coefficient relative to the
central active power is
\[
  \frac{R(s^\nu)}{s^{\nu+1}}
  =s\sin(s^{-\nu\gamma}).
\]
Because \(\nu\gamma>1\), its derivative has no continuous extension to
\(s=0\).  The smooth passage on the fixed stratum and the finite-event
structure away from it exclude this uncontrolled oscillation.

\subsection{Boundary-event curvature for prepared physical multiwells}

The four scalar and dynamical settings in this section differ by the
information retained from the passage:
\begin{center}
\small
\begin{tabular}{>{\raggedright\arraybackslash}p{0.20\textwidth}
                >{\raggedright\arraybackslash}p{0.34\textwidth}
                >{\raggedright\arraybackslash}p{0.33\textwidth}}
\toprule
Setting & Available structure & Cyclicity mechanism \\
\midrule
Fixed order-\(\nu\) stratum
  & Full ramified return with all lower contact jets zero
  & Cross-seam Hermite bound and separated-scale attainment \\
Complete exact cap
  & Positive-part area of the full polynomial contact jet
  & Coarea and quantitative sublevel growth through its discriminant \\
Prepared physical multiwell
  & Positive-part scalar flow over a finite-type contact graph
  & Same-sign entry--exit curvature, continued through births and mergers \\
Abstract value-controlled class
  & Weighted value asymptotic without event-derivative information
  & The counterexample in \cref{prop:value-only-no-cyclicity} shows that
    further derivative structure is necessary \\
\bottomrule
\end{tabular}
\end{center}

\paragraph{Prepared physical multiwell passage.}
Fix a common flow box
\([-T,T]\times(-I_*,I_*)\) and a parameter germ \(\theta=0\).  A
prepared physical multiwell passage of even order \(\nu\) consists of a
seam and a physical scalar flow of the form
\begin{equation}
  h_\theta(t,I)=g(t,I,\theta)(I-\Phi_\theta(t)),
  \qquad
  I'=(I-\Phi_\theta(t))_+B(t,I,\theta),
  \label{eq:prepared-physical-multiwell-definition}
\end{equation}
where \(g>0\), \(\Phi\) is jointly \(C^\nu\), \(B\) is jointly \(C^2\),
\(B(0,0,0)\neq0\), and \(\partial_t^\nu\Phi_\theta\) has a positive
uniform lower bound, with
\(\Phi_0(t)=at^\nu+o(|t|^\nu)\) for some \(a>0\).  The passage is a
prepared physical multiwell
\emph{return} when a smooth compatible transition outside the box closes
the same finite-time itinerary, all intervening sections are transverse,
and no orbit meets the incoming section before the designated return.

In the prepared box write
\[
  h_\theta(t,I)=g(t,I,\theta)(I-\Phi_\theta(t)),
  \qquad
  W_\theta=A_{\mathrm{geo}}\partial_t+B_{\mathrm{geo}}\partial_I,
  \qquad g>0.
\]
With
\begin{equation}
  B_{\mathrm{act}}(t,I,\theta)
  =\frac{gB_{\mathrm{geo}}}
  {1+g(I-\Phi_\theta(t))A_{\mathrm{geo}}},
  \label{eq:single-well-active-coefficient}
\end{equation}
the physical equation is exactly
\(I'=(I-\Phi_\theta(t))_+B_{\mathrm{act}}\).  Write \(B=B_{\mathrm{act}}\).

\begin{theorem}[Prepared Multiwell Boundary-Event Theorem]
\label{thm:prepared-multiwell-boundary-curvature}
\label{thm:boundary-event-curvature}
\label{prop:physical-single-well-curvature}
Let \(\nu\geq2\) be even, let \(\theta\) range near
\(0\in\R^m\), and suppose \(\Phi(t,\theta)\) is jointly \(C^\nu\) and
\(B(t,I,\theta)\) is jointly \(C^2\) on a common product
neighbourhood.  Write \(\Phi_\theta(t)=\Phi(t,\theta)\), and assume
\begin{equation}
  \Phi_0(t)=at^\nu+o(|t|^\nu),\quad a>0,
  \qquad
  \partial_t^\nu\Phi_\theta(t)\geq\kappa_*>0
  \quad (|t|\leq T).
  \label{eq:multiwell-finite-type}
\end{equation}
Consider
\begin{equation}
  I'=(I-\Phi_\theta(t))_+B(t,I,\theta),
  \qquad I(-T)=q,
  \qquad B(0,0,0)=\beta\neq0,
  \label{eq:multiwell-physical-ode}
\end{equation}
and put
\begin{equation}
  D(q,\theta)=I(T;q,\theta)-q,\qquad
  \tau_\theta=\min_{|t|\leq T}\Phi_\theta(t).
  \label{eq:multiwell-threshold-correction}
\end{equation}
After shrinking the common box and parameter germ, there is \(\rho_0>0\)
such that \(|\tau_\theta|<\rho_0\), every trajectory with
\(|q|+\|\theta\|<\rho_0\) is inactive at both \(t=-T\) and \(t=T\), and
the following statements hold.
\begin{enumerate}
\item For \(|q|<\rho_0\), the correction vanishes for
  \(q\leq\tau_\theta\).  On \(|q|<\rho_0\) it is \(C^1\) in \(q\), with \(D_q\) jointly
  continuous in \((q,\theta)\).  For each fixed \(\theta\), there is a set
  \(E_\theta\subset[\tau_\theta,\rho_0)\) of at most \(\nu-1\) levels that
  contains \(\tau_\theta\), such that \(D\) is \(C^2\) in \(q\) off
  \(E_\theta\).  In particular, \(D\) is smooth on the open inactive side
  and \(C^2\) at every regular active level.
\item At a regular active level, write
  \begin{equation}
    \mathcal A(q,\theta)
    =\{t:I(t;q,\theta)>\Phi_\theta(t)\}
    =\bigcup_{k=1}^{N}(a_k,b_k),
    \label{eq:multiwell-active-union}
  \end{equation}
  where \(a_k\) is an entry and \(b_k\) is the following exit.  Then
  \begin{equation}
    1\leq N\leq\frac{\nu}{2}.
    \label{eq:multiwell-component-bound}
  \end{equation}
  With
  \(\mathcal F=(I-\Phi_\theta)B\) and \(J=\partial_qI>0\),
  \begin{align}
    D_{qq}=J(T)\Bigg[&
    \sum_{k=1}^{N}\left\{
      \frac{B(b_k,I(b_k),\theta)J(b_k)}{\Phi_\theta'(b_k)}
      -\frac{B(a_k,I(a_k),\theta)J(a_k)}{\Phi_\theta'(a_k)}
    \right\}
    \notag\\[-2mm]
    &+\int_{\mathcal A(q,\theta)}
      \mathcal F_{II}(t,I(t),\theta)J(t)\dd t\Bigg].
    \label{eq:physical-passage-curvature-formula}
  \end{align}
\item There are \(\rho_1\in(0,\rho_0]\) and constants
  \(b_0,C_2>0\) such that \(\sgn B=\sgn\beta\), \(|B|\geq b_0\), and
  \(|\mathcal F_{II}|\leq C_2\) on the resulting common box.  For
  \(0<\rho\leq\rho_1\), put
  \begin{align}
    L(\rho)
    &=\sup\{\operatorname{meas}\mathcal A(q,\theta):
      |q|+\|\theta\|\leq\rho\},
      \label{eq:multiwell-active-length-modulus}\\
    \omega(\rho)
    &=\sup\{|\Phi_\theta'(c)|:c\text{ is a regular event},
      \ |q|+\|\theta\|\leq\rho\},
      \label{eq:multiwell-event-slope-modulus}
  \end{align}
  with the supremum of an empty event set defined to be zero.  Both
  moduli tend to zero with \(\rho\).  At every regular active point with
  \(|q|+\|\theta\|\leq\rho\), the event set is nonempty and
  \(\omega(\rho)>0\), and the quantitative curvature estimate
  \begin{equation}
    \sgn(\beta)D_{qq}(q,\theta)
    \geq\frac12\left(
      \frac{b_0}{\omega(\rho)}-2C_2L(\rho)\right)
    \label{eq:physical-passage-quantitative-curvature}
  \end{equation}
  holds.  Moreover, \(D,D_q\to0\) uniformly as
  \((q,\theta)\to(0,0)\).  Consequently, for every
  \(\varepsilon,M>0\), there is \(\rho\in(0,\rho_1]\) such that
  \(|D|+|D_q|<\varepsilon\) whenever
  \(|q|+\|\theta\|<\rho\), and at every regular active point there,
  \begin{equation}
    \sgn(\beta)D_{qq}>M.
    \label{eq:physical-passage-curvature-sign}
  \end{equation}
  Across the exceptional levels this becomes the global increment bound
  \begin{equation}
    \sgn(\beta)\bigl(D_q(q_2,\theta)-D_q(q_1,\theta)\bigr)
    \geq M(q_2-q_1)
    \label{eq:multiwell-derivative-increment}
  \end{equation}
  whenever \(\tau_\theta<q_1<q_2\) and
  \(|q_i|+\|\theta\|<\rho\) for \(i=1,2\).
\end{enumerate}
Thus, on a sufficiently small common neighbourhood,
\(\sgn(\beta)D_q(\cdot,\theta)\) is strictly increasing throughout the
active-level interval \(q>\tau_\theta\), including across levels at which
active components are born or merge.
\end{theorem}

\paragraph{Why multiple wells reinforce the endpoint mechanism.}
At every regular entry, \(\Phi_\theta'(a_k)<0\); at every regular exit,
\(\Phi_\theta'(b_k)>0\).  After shrinking, \(B\) has the fixed sign of
\(\beta\).  Consequently every summand in
\eqref{eq:physical-passage-curvature-formula} has that sign, independently
of the number or ordering of the active intervals.  The sum therefore
grows as all event slopes collapse, while the integral is bounded by the
total active duration.  At a birth or merger the transverse formula itself
is unavailable.  The difference-quotient argument in
Appendix~\ref{app:boundary-events} proves continuity of \(D_q\) there and
then joins the signed chamberwise estimates.  This derivative structure is
what the value-only family in \cref{prop:value-only-no-cyclicity} lacks.

The prepared pure tilt
\begin{equation}
  \Phi_{p,\eta}(t)
  =a(p,\eta)t^\nu+b\eta t+t^{\nu+1}R(t,p,\eta),
  \qquad a(0,0)>0,\quad b\neq0,
  \label{eq:prepared-pure-linear-tilt}
\end{equation}
is a single-well special case of the multiwell theorem.  Indeed,
\[
  \partial_t^2\Phi_{p,\eta}(t)
  =
  t^{\nu-2}
  \left[\nu(\nu-1)a(p,\eta)+O(t)\right]
\]
uniformly for small \((p,\eta)\).  After shrinking the fixed box, the
bracket is positive; hence \(\partial_t\Phi_{p,\eta}\) is strictly
increasing and has exactly one zero.  Thus the theorem applies to the
generic centered linear contact-jet direction, with neutral return
parameters allowed in the smooth coefficients.

\begin{corollary}[Smooth regular transmission preserves multiwell curvature]
\label[corollary]{cor:multiwell-curvature-transmission}
\label[corollary]{cor:single-well-curvature-transmission}
Let \(\mathcal R_\theta\) be a smooth family of regular outgoing maps with
\(L_0=\mathcal R_0'(0)\neq0\), and put
\begin{equation}
  G(q,\theta)=\mathcal R_\theta(q+D(q,\theta))-
  \mathcal R_\theta(q).
  \label{eq:single-well-transmitted-correction}
\end{equation}
Then \(G\) is \(C^1\) in \(q\), and \(G=G_q=0\) on the inactive side
and at the threshold.  For every \(M>0\), after shrinking,
\begin{equation}
  \sgn(L_0\beta)
  \bigl(G_q(q_2,\theta)-G_q(q_1,\theta)\bigr)
  \geq M(q_2-q_1)
  \label{eq:transmitted-curvature-sign}
\end{equation}
whenever \(\tau_\theta<q_1<q_2\).  At every regular active level,
\(G_{qq}\) has sign \(L_0\beta\), and its modulus diverges toward the
central germ.
\end{corollary}

\begin{proof}
Apply \cref{prop:threshold-curvature-transmission}.  Its algebraic proof
appears with the event estimates in Appendix~\ref{app:boundary-events}.
\end{proof}

\begin{theorem}[Sharp physical cyclicity on the prepared multiwell unfolding]
\label{thm:physical-multiwell-cyclicity}
\label{thm:physical-tilt-cyclicity}
Let an actual return displacement over the complete prepared contact graph
\eqref{eq:prepared-cap-graph} have the exact decomposition
\begin{equation}
  \Delta(q,p,\mathbf u)=S(q,p,\mathbf u)+G(q,p,\mathbf u),
  \label{eq:physical-tilt-displacement}
\end{equation}
where \(S\) is the smooth reference \(X^-\)-displacement.  Let
\(D(q,\theta)\) be the passage correction from
\cref{thm:prepared-multiwell-boundary-curvature}, let
\(\mathcal R_\theta\) be its regular outgoing transition, and set
\begin{equation}
  G(q,\theta)=\mathcal R_\theta(q+D(q,\theta))-\mathcal R_\theta(q),
  \qquad L_0=\mathcal R_0'(0)\neq0.
  \label{eq:physical-multiwell-transmitted-correction}
\end{equation}
Here \(\theta=(p,\mathbf u)\), the family satisfies
\cref{thm:prepared-multiwell-boundary-curvature} on a common box, and in
\eqref{eq:prepared-cap-graph} the coefficient \(a\) and the smooth
remainder may also depend smoothly on the neutral-return parameter \(p\).
Assume
\begin{equation}
  \begin{aligned}
    \partial_q^jS(0,0,0)&=0 &&(0\leq j<n),\\
    c&=\frac{\partial_q^nS(0,0,0)}{n!}\neq0,
    &d&=L_0\beta\neq0.
  \end{aligned}
  \label{eq:physical-tilt-central-data}
\end{equation}
Then there are \(r_0>0\) and a neighbourhood \(V\) of \((p,\mathbf u)=(0,0)\)
such that, for every \((p,\mathbf u)\in V\),
\begin{equation}
  \#\{q:|q|<r_0,\ \Delta(q,p,\mathbf u)=0\}
  \leq
  \begin{cases}
    n,&cd>0,\\
    n+1,&cd<0.
  \end{cases}
  \label{eq:physical-tilt-cyclicity-bound}
\end{equation}
All zeros in this interval are isolated.  If the central smooth coefficient
map
\begin{equation}
  p\longmapsto\left(S(0,p,0),\partial_qS(0,p,0),\ldots,
  \frac{\partial_q^{n-1}S(0,p,0)}{(n-1)!}\right)
  \label{eq:physical-multiwell-rank-map}
\end{equation}
has rank \(n\), and the central slice \(\mathbf u=0\) is a fixed-contact
actual-return family satisfying \(\mathrm{(FR)}_\nu\), the applicable bound
is sharp on that slice; its simple periodic orbits have alternating stability.
\end{theorem}

\begin{proof}
Let \(\tau(p,\mathbf u)\) be the moving minimum.  On the inactive side \(G=0\),
so \(\sgn(c)\partial_q^n\Delta>0\) after shrinking.  On the active side,
\(S_{qq}\) is uniformly bounded.  Choose the modulus in
\cref{cor:multiwell-curvature-transmission} larger than that bound.  It
follows directly from the derivative increment that
\(\sgn(d)\Delta_q\) is strictly increasing across every regular chamber
and event-discriminant level.  The displacement is \(C^1\) across the
moving threshold, so
\cref{cor:moving-threshold-cyclicity} proves the bound and isolation.
The central rank statement follows from
\cref{cor:fixed-order-circuit-recognition,thm:sharp-cyclicity,cor:stability}.
\end{proof}

The exact-cap and physical theorems use different inputs.  Coarea controls
the exact polynomial cap directly.  The physical theorem instead uses the
prepared finite-type condition, scalar-flow order, fixed sign of the two-field
mismatch, and \(C^1\) continuation across the finite event discriminant.
The oscillatory value-only counterexample lies outside this physical class
and identifies the boundary-event argument as the additional input required
beyond the weighted value estimate.

\section{Closed planar realization}
\label{sec:closed-realization}
\subsection{A closed planar realization for every even
  \texorpdfstring{\(\nu\geq4\)}{nu >= 4}}
\label{sec:closed-polynomial-realization}

The preceding scalar and contact-incidence theorems distinguish an abstract
root bound from an actual return realization.  We now establish
attainability and sharpness in an explicit closed planar family.  Every
return below arises directly
as a finite-time Poincar\'e map of a continuous piecewise-polynomial vector
field.

Fix \(n\geq2\), an integer \(\ell\geq2\), and put
\begin{equation}
  \nu=2\ell,
  \qquad
  \varrho=x^2+y^2-1,
  \qquad
  h_\ell(x,y)=\varrho-(1-x)^\ell.
  \label{eq:closed-polynomial-seam}
\end{equation}
Let
\[
  R(x,y)=(-y,x),
  \qquad
  V(x,y)=(x,y),
\]
and, for
\[
  \lambda=(\lambda_0,\ldots,\lambda_{n-1})\in\R^n,
\]
define
\begin{equation}
  G_\lambda(x,\varrho)
  =
  \gamma \varrho^n
  +\sum_{j=1}^{n-1}\lambda_j\varrho^j
  +\lambda_0(1-x)^\ell,
  \qquad \gamma\neq0.
  \label{eq:closed-polynomial-radial-factor}
\end{equation}
For a fixed \(\beta\neq0\), consider the two polynomial extensions
\begin{equation}
  X^-_\lambda
  =R+\frac12G_\lambda V,
  \qquad
  X^+_\lambda
  =X^-_\lambda+\frac{\beta}{2}h_\ell V,
  \label{eq:closed-polynomial-branches}
\end{equation}
and the physical continuous piecewise-polynomial field
\begin{equation}
  X_\lambda
  =
  \begin{cases}
    X^-_\lambda,&h_\ell\leq0,\\
    X^+_\lambda,&h_\ell\geq0.
  \end{cases}
  \label{eq:closed-polynomial-physical-field}
\end{equation}
The two extensions agree on \(h_\ell=0\).  Equivalently,
\(X_\lambda=X^-_\lambda+(\beta/2)(h_\ell)_+V\), so the physical field is
locally Lipschitz and has unique local solutions.

The seam geometry directly displays the first nonquadratic case.  For
\(\ell=2\), the seam is the parabola
\(x=1-y^2/2\).  On the unit circle,
\[
  h_2(x,y)=-(1-x)^2,
\]
so the circle stays in \(h_2<0\) except at \(z=(1,0)\).  The local gap in
\cref{fig:closed-quartic-geometry} makes the order-four contact explicit.

\begin{theorem}[Closed polynomial realization of even-order osculation]
\label{thm:closed-polynomial-realization}
For every \(n\geq2\) and every even \(\nu=2\ell\geq4\), the family
\eqref{eq:closed-polynomial-physical-field} has the following properties.

\begin{enumerate}
\item At \(\lambda=0\), the unit circle is a periodic orbit of period
  \(2\pi\).  It has one grazing point
  \(z=(1,0)\), of contact order \(\nu\), with
  \begin{equation}
    \kappa_\nu=\frac{\nu!}{2^\ell},
    \qquad
    \beta_{\mathrm{gr}}=\beta.
    \label{eq:closed-polynomial-contact-data}
  \end{equation}
  The same point remains an order-\(\nu\) tangency of the reference
  \(X^-\)-family for all sufficiently small \(\lambda\).

\item A fixed incoming ray before \(z\) carries a smooth
  parameter-dependent coordinate \(q\), with \(q=0\) on the reference
  \(X^-_\lambda\)-orbit through \(z\), hence on the central periodic orbit
  when \(\lambda=0\).  For every sufficiently small \(\lambda\), the physical
  first-return map is a finite-time Poincar\'e map, and its
  displacement has the form
  \begin{equation}
    \Delta(q,\lambda)
    =S(q,\lambda)
    +q_+^{1+1/\nu}
      K(q_+^{1/\nu},q,\lambda).
    \label{eq:closed-polynomial-displacement}
  \end{equation}
  On the central fibre,
  \begin{equation}
    S(q,0)=2\pi\gamma q^n+O(q^{n+1}),
    \qquad
    K(0,0,0)=
    \frac{2\nu}{\nu+1}\sqrt2\,\beta.
    \label{eq:closed-polynomial-leading-data}
  \end{equation}

\item If
  \[
    a_k(\lambda)=\frac1{k!}\partial_q^kS(0,\lambda),
    \qquad 0\leq k<n,
  \]
  then the coefficient map
  \(\lambda\mapsto(a_0,\ldots,a_{n-1})\) has rank \(n\) at the
  origin.  More precisely, its derivative is triangular, with diagonal
  \begin{equation}
    C_\ell,2\pi,\ldots,2\pi,
    \qquad
    C_\ell=
    \int_0^{2\pi}(1-\cos\theta)^\ell\dd\theta
    =\frac{2\pi(2\ell)!}{2^\ell(\ell!)^2}>0.
    \label{eq:closed-polynomial-rank-diagonal}
  \end{equation}

\item The selected local periodic-orbit cyclicity in this physical
  family is therefore
  \begin{equation}
    \begin{cases}
      n,&\gamma\beta>0,\\
      n+1,&\gamma\beta<0.
    \end{cases}
    \label{eq:closed-polynomial-cyclicity}
  \end{equation}
  In each sign case, the applicable value is attained by actual simple
  periodic orbits for parameter sequences \(\lambda\to0\).  These cycles
  are hyperbolic after the germ is shrunk, and their stability alternates
  along the return section.
\end{enumerate}
\end{theorem}

\begin{proof}[Proof outline]
The identity \(\dot\theta=1\) gives the time-\(2\pi\) return, the seam
gap \(-(1-\cos t)^\ell\) gives the unique order-\(2\ell\) contact, and the
\(\lambda\)-variations give a triangular jet matrix with diagonal
\(C_\ell,2\pi,\ldots,2\pi\).  The common-domain and variational details are
in \cref{app:closed-realization}; the passage and sharp scalar theorems then
give the stated cycles.
\end{proof}

\begin{corollary}[Quartic grazing benchmarks]
\label[corollary]{cor:quartic-polynomial-benchmarks}
For \(\ell=2\),
\begin{equation}
  h_2=y^2+2x-2,
  \qquad
  \kappa_4=6,
  \qquad
  c=2\pi\gamma,
  \qquad
  d=\frac{8\sqrt2}{5}\beta,
  \label{eq:quartic-polynomial-data}
\end{equation}
and the rank determinant is
\begin{equation}
  3\pi(2\pi)^{n-1}.
  \label{eq:quartic-polynomial-rank}
\end{equation}
For \(n=2\), \eqref{eq:closed-polynomial-branches} is a degree-five
piecewise-polynomial family realizing two or three local cycles according
as \(\gamma\beta>0\) or \(\gamma\beta<0\); for \(n=3\), its degree is
seven and it realizes three or four.  This explicit algebraic benchmark
realizes both the grazing geometry and the sharp scalar configurations in a
closed planar vector field with finite-time Poincar\'e maps.  Its unfolding
remains within the fixed-order-\(\nu\)
stratum; general transverse perturbations may split the high-order
tangency into several lower-order events.
\end{corollary}

\begin{theorem}[Contact-jet-transverse extension of the closed family]
\label{thm:closed-contact-jet-transverse-extension}
Let \(n\geq2\), \(\ell\geq2\), and \(\nu=2\ell\).  For
\(\eta=(\eta_1,\ldots,\eta_{\nu-2})\), replace
\eqref{eq:closed-polynomial-seam} by
\begin{equation}
  h_{\ell,\eta}(x,y)
  =\varrho-(1-x)^\ell
   +\sum_{j=1}^{\nu-2}\eta_jy^j,
  \label{eq:closed-transverse-seam}
\end{equation}
keep \(X^-_\lambda\) from
\eqref{eq:closed-polynomial-branches}, and set
\begin{equation}
  X^+_{\lambda,\eta}
  =X^-_\lambda+\frac{\beta}{2}h_{\ell,\eta}V,
  \qquad
  X_{\lambda,\eta}
  =X^-_\lambda+\frac{\beta}{2}
    (h_{\ell,\eta})_+V.
  \label{eq:closed-transverse-field}
\end{equation}
Then the following statements hold near
\((z,\lambda,\eta)=((1,0),0,0)\).

\begin{enumerate}
\item The physical field is continuous and locally Lipschitz, and a fixed
  ray carries time-\(2\pi\) Poincar\'e maps on a common state and
  parameter domain.

\item Define the full contact map
  \begin{equation}
    \mathcal J_\nu^{\mathrm{cl}}(z',\lambda,\eta)
    =\bigl(
      h_{\ell,\eta},
      X^-_\lambda h_{\ell,\eta},
      \ldots,
      (X^-_\lambda)^{\nu-1}h_{\ell,\eta}
    \bigr)(z').
    \label{eq:closed-contact-map}
  \end{equation}
  Its derivative with respect to the two state variables and the
  \(\nu-2\) parameters \(\eta\) is invertible at the base point.  In
  particular, the hypotheses of
  Definition~\ref{def:ambient-transverse-contact} and
  Theorem~\ref{thm:ambient-contact-incidence} hold.
  Hence
  the order-\(\nu\) contact stratum in the physical parameter
  space is locally \(\eta=0\) and has codimension \(\nu-2\).  Thus
  \(\eta\) supplies a minimal transverse contact-jet slice for this
  selected incidence germ after the contact point is allowed to move; the
  two omitted jet directions are
  precisely displacement normal to the seam and recentering along the
  reference flow.

\item Let \(a_0(\lambda),\ldots,a_{n-1}(\lambda)\) be the smooth-reference
  coefficients from \cref{thm:closed-polynomial-realization}.  The combined
  map
  \begin{equation}
    (\lambda,\eta)\longmapsto
    \bigl(a_0(\lambda),\ldots,a_{n-1}(\lambda),
          \eta_1,\ldots,\eta_{\nu-2}\bigr)
    \label{eq:closed-total-codimension-map}
  \end{equation}
  has rank \(n+\nu-2\).  Consequently, the simultaneous order-\(\nu\)
  grazing and order-\(n\) neutral-return locus has physical codimension
  \(n+\nu-2\) in this explicit family, in agreement with
  \cref{thm:total-physical-codimension}.

\item Throughout the same small \((\lambda,\eta)\)-domain, the number of
  periodic orbits whose section coordinates remain near zero is at most
  \begin{equation}
    \begin{cases}
      n,&\gamma\beta>0,\\
      n+1,&\gamma\beta<0.
    \end{cases}
    \label{eq:closed-transverse-physical-bound}
  \end{equation}
  Every such fixed point is isolated.  The applicable bound is attained
  on \(\eta=0\), where the simple periodic orbits alternate in stability.
\end{enumerate}

On the central stratum \(\eta=0\), all conclusions of
\cref{thm:closed-polynomial-realization} remain unchanged and provide the
sharp witnesses.  The upper bound in
\eqref{eq:closed-transverse-physical-bound} applies across the complete
prepared contact-jet slice, including parameters with several active
crossing intervals.
\end{theorem}

\begin{proof}[Proof outline]
The state normal, flow tangent, and \(\eta\)-directions form a triangular
contact-jet matrix with diagonal
\[
  1,1!,\ldots,(\nu-2)!,-\frac{\nu!}{2^\ell}.
\]
The return-coefficient block is the preceding invertible \(\lambda\)-block,
so the combined derivative is block diagonal.  Appendix
\ref{app:closed-realization} proves the common return domain and rank;
local preparation and \cref{thm:physical-multiwell-cyclicity} give the
fourth conclusion.
\end{proof}

\section{Discussion and scope}
\label{sec:conclusion}

The results separate two roles: the neutral order \(n\) controls the sharp
cross-seam orbit count, while the contact order \(\nu\) controls physical
codimension, passage scale, and transverse weights.  Rank-\(n\) unfoldings
attain the \(n\) or \(n+1\) bound, and contact incidence gives total
codimension \(n+\nu-2\).

Two reusable tools connect the flow analysis to the zero count.
The Cross-Seam theorem bounds roots of two branches that share a finite seam
jet and allows a full smooth ramified remainder; it applies whenever a
dynamical or geometric reduction produces this scalar structure.  The
Prepared Multiwell Boundary-Event Theorem treats a positive-part scalar flow
over a finite-type contact graph.  Its entry and exit terms have a common
sign, and its first variation continues through births and mergers.  This
gives a general criterion for carrying derivative-based root bounds across
changing event topology.

For the two-field flow, localized passage and regular transmission embed
these scalar results in an actual first-return map.  The weighted cap
connects exponent \(1+1/\nu\) to the generic quadratic-contact exponent
\(3/2\), and the multiwell theorem preserves the sharp total through
repeated entries and exits.  For every even \(\nu\geq4\), the closed
polynomial family realizes all hypotheses and sharp configurations in one
planar system.

\subsection*{Scope of the multiwell conclusion}

The multiwell theorem assumes a uniform finite-type condition, a nonvanishing
active coefficient of fixed sign, and a passage that is inactive at both
boundaries of a common flow box.  These hypotheses make every endpoint
reinforce the same signed curvature.  They also identify extensions worth
studying: sign-changing
transmission between wells and contact graphs without uniform finite type
require a different organizing mechanism.  The oscillatory \(C^1\) family in
\cref{prop:value-only-no-cyclicity} identifies event structure, beyond
weighted value control, as essential for finite cyclicity.

\section*{Declarations}

\noindent\textbf{Acknowledgements and generative-AI disclosure.}
OpenAI Codex (GPT-5, accessed August 2026; immutable snapshot identifier
unavailable) assisted with language and structural editing, cross-reference
and symbolic consistency checks, literature searches, and figure-code and
layout editing.  The author verified the mathematical statements, citations,
figures, and resulting text and accepts full responsibility for the article.
\par\smallskip
\noindent\textbf{Funding.}
This research received no specific grant.
\par\smallskip
\noindent\textbf{Competing interests.}
The author declares no competing interests.
\par\smallskip
\noindent\textbf{Data and code availability.}
No empirical data were used.  The source files, symbolic checks, and
figure-generation scripts are available in the
\href{https://github.com/h-lu/-RAMIFIED-SEAM-SINGULARITIES}{project
repository}.  All analytical arguments are contained in the article.

\clearpage
\appendix
\section{Uniform Hermite and separated-scale estimates}
\label{app:uniform-hermite}

This appendix proves the finite-regularity estimates behind
\cref{thm:cross-seam-hermite-zero}, the uniform
ramified shape estimate, and the two-scale persistence argument used in
\cref{thm:rank-lifting-root-realization}.

\subsection{One-sided signs and the confluent node}

\begin{lemma}[Positive-side seam signs]
\label[lemma]{lem:positive-seam-signs}
Let \(g\in C^1([0,r])\) and suppose that \(g'\) is strictly increasing.
Then:
\begin{enumerate}
\item \(g\) has at most two distinct zeros in \([0,r]\);
\item if \(0<y_1<y_2\leq r\) are zeros, then
  \(g(0)>0\) and \(g'(0)<0\);
\item if \(0\) and \(y>0\) are zeros, then \(g'(0)<0\);
\item if \(g(0)>0\) and \(g'(0)>0\), then \(g\) has no positive zero.
\end{enumerate}
\end{lemma}

\begin{proof}
Three distinct zeros would give two distinct zeros of \(g'\), contradicting
strict monotonicity.  If \(0<y_1<y_2\) are zeros, the mean value theorem
gives \(\xi\in(0,y_1)\) and \(\zeta\in(y_1,y_2)\) with
\[
  g'(\xi)=-\frac{g(0)}{y_1},
  \qquad
  g'(\zeta)=0.
\]
Since \(\xi<\zeta\), strict increase gives \(g'(\xi)<0\), hence
\(g(0)>0\), and then \(g'(0)<g'(\xi)<0\).  The same argument on
\([0,y]\) proves the third assertion.  For the fourth assertion,
\(g'(q)>g'(0)>0\) for every \(q>0\), so a positive seam value remains
positive.
\end{proof}

\begin{lemma}[\(n\) inactive roots fix the seam signs]
\label[lemma]{lem:n-negative-signs}
Assume \eqref{eq:left-shape}.  If \(f\) has \(n\) distinct roots
\(x_1<\cdots<x_n<0\), then
\begin{equation}
  \sgn f(0)=\eta,
  \qquad
  \sgn f'(0)=\eta.
  \label{eq:n-negative-seam-signs}
\end{equation}
\end{lemma}

\begin{proof}
The divided-difference mean value theorem
\cite[Proposition~43 and the Schwarz corollary that follows it]{deBoor2005}
gives
\[
  [x_1,\ldots,x_n,0]f=\frac{f^{(n)}(\xi)}{n!}
\]
for some \(\xi\in(x_1,0)\).  Since the values at the \(x_i\) vanish,
\[
  [x_1,\ldots,x_n,0]f
  =\frac{f(0)}{\prod_{i=1}^n(0-x_i)}.
\]
The denominator is positive, proving the first sign.  Rolle's theorem
supplies \(\xi_i\in(x_i,x_{i+1})\) with \(f'(\xi_i)=0\).  Apply the same
argument to \(f'\) at \(\xi_1,\ldots,\xi_{n-1},0\) to obtain
\[
  \frac{f'(0)}{\prod_{i=1}^{n-1}(0-\xi_i)}
  =\frac{f^{(n)}(\zeta)}{(n-1)!}
\]
for some \(\zeta\in(\xi_1,0)\).  This proves the second sign under the
stated finite regularity.
\end{proof}

\begin{lemma}[Confluent seam divided difference]
\label[lemma]{lem:confluent-seam}
Assume \eqref{eq:left-shape} and suppose that \(f\) has \(n-1\) distinct
roots \(x_1<\cdots<x_{n-1}<0\).  Set
\[
  a=f(0),\quad b=f'(0),\quad
  Q=\prod_{i=1}^{n-1}(-x_i),\quad
  L=\sum_{i=1}^{n-1}\frac1{-x_i}.
\]
Then \(\sgn(b-La)=\eta\).
\end{lemma}

\begin{proof}
Let \(H\) be the polynomial of degree at most \(n\) satisfying
\(H(x_i)=0\), \(H(0)=a\), and \(H'(0)=b\).  It has the form
\[
  H(x)=\prod_{i=1}^{n-1}(x-x_i)(ux+v),
\]
and the two seam conditions give
\begin{equation}
  v=\frac aQ,
  \qquad
  u=\frac{b-La}{Q}.
  \label{eq:Hermite-leading-coefficient}
\end{equation}
The function \(f-H\) vanishes at the \(n-1\) nodes and to order at least
two at zero.  This is the confluent repeated-node case of the same
mean-value principle
\cite[the Schwarz corollary following Proposition~43]{deBoor2005};
repeated Rolle gives
\(\xi\in(x_1,0)\) with
\[
  0=(f-H)^{(n)}(\xi)=f^{(n)}(\xi)-n!u.
\]
Thus \(u\) has sign \(\eta\), and \(Q>0\) completes the proof.
\end{proof}

\begin{proof}[Proof of \cref{thm:cross-seam-hermite-zero}]
Replace \(f\) by \(g=\sigma f\).  This leaves the zero set unchanged,
makes \(g'\) strictly increasing on the active branch, and gives
\(\sgn g^{(n)}=\eta\sigma=: \tau\) on the inactive branch.  Repeated
Rolle and \cref{lem:positive-seam-signs} give
\begin{equation}
  N_-(g)+z_0(g)\leq n,
  \qquad
  N_+(g)\leq
  \begin{cases}
    2,&z_0(g)=0,\\
    1,&z_0(g)=1.
  \end{cases}
  \label{eq:separate-estimates}
\end{equation}

Suppose first that \(\tau=-1\).  If \(g(0)=0\), the separate estimates
give a total at most \(n+1\).  If \(g(0)\neq0\) and
\(N_-(g)\leq n-1\), the same follows.  In the remaining case
\(N_-(g)=n\), \cref{lem:n-negative-signs} gives \(g(0)<0\), whereas two
active zeros would give \(g(0)>0\).  Hence \(N_+(g)\leq1\).

Now let \(\tau=1\).  If \(g(0)\neq0\), only two configurations could
exceed \(n\).  If \(N_-(g)=n\), then
\cref{lem:n-negative-signs} gives \(g(0)>0\) and \(g'(0)>0\), excluding
every active zero.  If \(N_-(g)=n-1\) and there were two active zeros,
\cref{lem:positive-seam-signs} would give \(a=g(0)>0\) and
\(b=g'(0)<0\).  But \cref{lem:confluent-seam} requires \(b-La>0\),
while \(L>0\) makes it negative.  If instead zero itself is a root, a
total above \(n\) requires \(n-1\) inactive roots and one active root.
The confluent lemma with \(a=0\) gives \(g'(0)>0\), whereas the seam and
active roots force \(g'(0)<0\).

Finally, an inactive zero interval contradicts the signed \(n\)th
derivative.  An active zero interval makes \(g'\) vanish at two points and
contradicts strict increase.  An interval crossing the seam contains an
interval of one of these types.
\end{proof}

\subsection{Uniform shape of the full ramified remainder}

\begin{proof}[Proof of \cref{prop:shape-control}]
The inactive assertion follows by continuity from
\(\partial_q^nS(0,0)=n!c\).  For the active assertion put
\[
  \widehat K(s,p)=K(s,s^\nu,p),
  \qquad
  M(s,p)=\Delta_q(s^\nu,p).
\]
Because \(\alpha_\nu>1\), the ramified summand and its \(q\)-derivative
tend to zero at the seam, so \(\Delta\) is \(C^1\) in \(q\) and
\(M(0,p)=S_q(0,p)\).  For \(s>0\),
\[
  M(s,p)=S_q(s^\nu,p)
  +\frac{\nu+1}{\nu}s\widehat K(s,p)
  +\frac1\nu s^2\widehat K_s(s,p).
\]
Differentiation gives
\[
  \begin{aligned}
  M_s(s,p)={}&\nu s^{\nu-1}S_{qq}(s^\nu,p)
  +\frac{\nu+1}{\nu}\widehat K(s,p)\\
  &+\frac{\nu+3}{\nu}s\widehat K_s(s,p)
  +\frac1\nu s^2\widehat K_{ss}(s,p).
  \end{aligned}
\]
The assumed regularity makes this continuous at the origin, where
\(M_s(0,0)=(\nu+1)d/\nu\).  Shrinking uniformly in \((s,p)\) proves
\eqref{eq:positive-uniform-sign}.
\end{proof}

\subsection{Lifted coefficient paths and the active scale}

\begin{proof}[Proof of \cref{thm:rank-lifting-root-realization}]
After restricting to an \(n\)-dimensional parameter slice, the submersion
theorem gives a local right inverse
\[
  \pi:(\R^n,0)\longrightarrow(\R^m,0),
  \qquad A(\pi(\bfa))=\bfa.
\]
Take
\(p_t=\pi(\widehat a_0(t),\ldots,\widehat a_{n-1}(t))\).
Since \(\widehat a_j(t)=O(t^{n-j})\) for \(j<n\), this path is \(O(t)\).

Put \(c(p)=\partial_q^nS(0,p)/n!\).  Uniform Taylor expansion gives
\begin{equation}
  S(q,p_t)
  =\mathcal P_{n,t}(q)
  +(c(p_t)-c)q^n+q^{n+1}R_0(q,p_t),
  \label{eq:sharp-Taylor}
\end{equation}
with \(R_0\) and its first derivative bounded.  If
\(Q(x)=\prod_{j=1}^n(x+r_j)\), then
\begin{align*}
  t^{-n}S(tx,p_t)
  &=cQ(x)+(c(p_t)-c)x^n+tx^{n+1}R_0(tx,p_t),\\
  \frac{\dd}{\dd x}\bigl[t^{-n}S(tx,p_t)\bigr]
  &=cQ'(x)+n(c(p_t)-c)x^{n-1}+t x^nR_1(tx,p_t).
\end{align*}
The rescaled functions converge in \(C^1\) on compact \(x\)-sets to
\(cQ\).  Its simple roots persist, giving
\eqref{eq:persistent-negative-roots}.  The ramified term vanishes on this
branch, so these are roots of the full displacement.

On the active scale, set \(q=t^{\gamma_{n,\nu}}y\), \(y>0\), and put
\[
  Z_t(y)=
  \left(t^{n/(\nu+1)}y^{1/\nu},
        t^{\gamma_{n,\nu}}y,p_t\right).
\]
The constant coefficient of \(\mathcal P_{n,t}\) is
\(c(\prod_jr_j)t^n\).  For \(j\geq1\),
\[
  t^{-n}\widehat a_j(t)q^j
  =O\!\left(t^{j(\gamma_{n,\nu}-1)}\right)\longrightarrow0.
\]
The varying \(q^n\) coefficient and Taylor remainder also vanish after
division by \(t^n\), while
\[
  t^{-n}q^{\alpha_\nu}K(q^{1/\nu},q,p_t)
  =y^{\alpha_\nu}K(Z_t(y))
  \longrightarrow d y^{\alpha_\nu}.
\]
Thus the rescaled displacement converges locally uniformly on \(y>0\) to
\eqref{eq:positive-limit-function}.

For derivative control define
\(F_t(y)=t^{-n}\Delta(t^{\gamma_{n,\nu}}y,p_t)\).  Then
\begin{equation}
  F_t'(y)=t^{\gamma_{n,\nu}-n}
  \Delta_q(t^{\gamma_{n,\nu}}y,p_t).
  \label{eq:positive-derivative-scaling}
\end{equation}
The differentiated smooth terms converge by \eqref{eq:sharp-Taylor}.  The
ramified contribution is \(y^{\alpha_\nu}K(Z_t(y))\), whose derivative is
\begin{align}
  &\alpha_\nu y^{\alpha_\nu-1}K(Z_t(y))
  \notag\\
  &\quad+y^{\alpha_\nu}\left[
    \frac{t^{n/(\nu+1)}}{\nu}y^{1/\nu-1}\partial_1K(Z_t(y))
    +t^{\gamma_{n,\nu}}\partial_2K(Z_t(y))
  \right].
  \label{eq:ramified-scaled-derivative}
\end{align}
On every \(I\Subset(0,\infty)\), the powers of \(y\) are bounded and
\(K\in C^2\) gives
\[
  F_t'\longrightarrow
  \alpha_\nu d y^{\alpha_\nu-1}
  \quad\text{uniformly on }I.
\]
The factor \(y^{1/\nu-1}\) may be singular at \(y=0\), so the argument uses
uniform convergence on compact subsets of \((0,\infty)\).

When \(cd<0\), the limiting function has the unique simple positive root
\[
  y_0=
  \left(\left|\frac cd\right|\prod_{j=1}^nr_j
  \right)^{\nu/(\nu+1)}.
\]
Uniform \(C^1\) convergence on a compact interval about \(y_0\) gives a
unique simple nearby root
\(q_{\mathrm{ram}}(t)=t^{n\nu/(\nu+1)}(y_0+o(1))\).
Equation \eqref{eq:positive-derivative-scaling} gives its derivative sign
\(\sgn d\), proving all assertions.
\end{proof}

\section{Contact-jet invariance and persistence}
\label{app:contact-strata}

This appendix supplies the differential-topological details behind the
contact-stratum statements in \cref{sec:ambient-codimension}.  The selected
parameter projection used in the ambient codimension theorem requires more
than a regular-level-set calculation.

\subsection{The selected incidence projection}

\begin{proof}[Proof of \cref{thm:ambient-contact-incidence}]
The regular-level-set theorem applied to
\eqref{eq:ambient-contact-jet} gives
\[
  \dim\mathcal I_\nu=m+2-\nu.
\]
The projection requires a separate argument because regularity of the
incidence manifold alone leaves the geometry of its image in parameter space
undetermined.

At the contact point, differentiation along the vector field gives
\[
  d_zH_k(z_0,p_0)\bigl[X_{p_0}(z_0)\bigr]
  =H_{k+1}(z_0,p_0).
\]
Consequently,
\begin{equation}
  D_z\mathcal J_\nu(z_0,p_0)
  \bigl[X_{p_0}(z_0)\bigr]
  =\bigl(0,\ldots,0,H_\nu(z_0,p_0)\bigr).
  \label{eq:app-contact-jet-flow-direction}
\end{equation}
Choose \(W\in T_{z_0}M\) with
\(d_zh_{p_0}(z_0)[W]\neq0\).  The first component of
\(D_z\mathcal J_\nu[W]\) is nonzero, whereas the first component in
\eqref{eq:app-contact-jet-flow-direction} vanishes.  These two images are
linearly independent.  Since \(M\) is two-dimensional,
\begin{equation}
  \rank D_z\mathcal J_\nu(z_0,p_0)=2;
  \label{eq:app-state-contact-jet-rank}
\end{equation}
in particular, \(D_z\mathcal J_\nu\) is injective.

Now let \((\xi,\eta)\in T_{(z_0,p_0)}\mathcal I_\nu\) lie in the kernel of
\(D\pi_P\).  Then \(\eta=0\), and tangency to the incidence manifold gives
\[
  D_z\mathcal J_\nu(z_0,p_0)[\xi]=0.
\]
By \eqref{eq:app-state-contact-jet-rank}, \(\xi=0\).  Hence
\(D(\pi_P|_{\mathcal I_\nu})\) is injective at \((z_0,p_0)\), and remains
injective after shrinking.  The local immersion theorem provides a smaller
neighbourhood on which \(\pi_P|_{\mathcal I_\nu}\) is an embedding.  All
projection statements refer to this neighbourhood of the selected incidence
point.  Its image has dimension \(m+2-\nu\), so its codimension in \(P\) is
\(\nu-2\).  The inverse of this local embedding gives the smooth
contact-point selection \(p\mapsto z(p)\).
\end{proof}

\subsection{Triangular covariance}

\begin{proof}[Proof of \cref{lem:contact-jet-triangular-invariance}]
Leibniz' rule gives
\begin{equation}
  L_X^k(uh)
  =\sum_{j=0}^k\binom{k}{j}
    (L_X^{k-j}u)(L_X^jh),
  \label{eq:app-defining-function-triangularity}
\end{equation}
which is lower triangular in
\((h,L_Xh,\ldots,L_X^kh)\), with diagonal coefficient \(u\).
Inductively, the differential operator \((aL_X)^k\) has the form
\begin{equation}
  (aL_X)^k=a^kL_X^k+
  \sum_{j=1}^{k-1}b_{kj}L_X^j
  \label{eq:app-time-triangularity}
\end{equation}
for smooth coefficients \(b_{kj}\); the assertion follows by applying
\(aL_X\) once more.  Combining
\eqref{eq:app-defining-function-triangularity} and
\eqref{eq:app-time-triangularity} proves
\eqref{eq:contact-jet-triangular-law}.  At a zero of
\(\mathcal J_\nu\), differentiating
\(\widehat{\mathcal J}_\nu=\mathsf T_\nu\mathcal J_\nu\) eliminates the term
\((D\mathsf T_\nu)\mathcal J_\nu\), and the order-\(\nu\) identity follows
from the same triangular calculation.  Since every diagonal entry is
nonzero, all rank assertions are preserved.

Naturality of Lie derivatives under a diffeomorphism gives
\eqref{eq:contact-jet-coordinate-covariance}.  The map
\((z,p)\mapsto(\Phi_p(z),p)\), together with any chosen local change of
parameter coordinates, is a diffeomorphism of the total space and commutes
with parameter projection up to that coordinate change.  It therefore carries
the selected embedded image germ to a diffeomorphic image germ of the same
codimension.
\end{proof}

\subsection{Persistence of the selected branch}

\begin{proof}[Proof of \cref{prop:ambient-stratum-persistence}]
Because \(D\mathcal J_\nu(z_0,p_0)\) is onto, choose \(\nu\) local
coordinates among the two state and \(m\) parameter coordinates on which
the corresponding square derivative is invertible.  Treating all remaining
coordinates as free variables, the implicit function theorem writes the
incidence set as a graph.  Invertibility of this square derivative is open,
so the same graph construction applies to every sufficiently small
perturbation and, in a smooth finite-dimensional perturbation family,
depends smoothly on the extra parameter.

The inequalities \(X\neq0\), \(dh\neq0\), and \(H_\nu\neq0\) are open.
At every nearby incidence point they again imply, by the calculation in
\eqref{eq:app-contact-jet-flow-direction}, that the state derivative of the
perturbed contact jet has rank two.  The projection argument in
\cref{thm:ambient-contact-incidence} therefore applies uniformly after a
further restriction.  Equivalently, retain a fixed nonzero minor of the
differential of the projected incidence map and apply the
parameter-dependent constant-rank theorem; this gives the asserted smooth
dependence of the embedded image germ.  Finally, the sign of \(H_\nu\) is
open.  The \(C^\nu\) topology is sufficient because the contact jet and its
first differential, together with \(H_\nu\), involve derivatives through
order \(\nu\).
\end{proof}

\section{Localized active-region passage and weighted blow-up}
\label{app:passage}

This appendix proves the fixed-contact and transverse passage theorems by
localization, continuation, rescaling, and event selection.  The
fixed-contact theorem produces a full smooth ramified remainder.  The
transverse theorem gives a value estimate that remains uniform as the
number of active components changes.

\subsection{Fixed even contact: full ramified smoothness}

\begin{proof}[Proof of \cref{prop:even-order-grazing-comparison}]
A parameter-dependent \(X_p^-\)-flow box gives coordinates \((t,I)\), with
\(z_p=(0,0)\), in which \(X_p^-=\partial_t\) and
\begin{equation}
  h_p(t,I)
  =I-\frac{\kappa_\nu(p)}{\nu!}t^\nu
  +t^{\nu+1}R_0(t,I,p)+tI R_1(t,I,p)+I^2R_2(t,I,p),
  \label{eq:grazing-flow-box}
\end{equation}
for jointly smooth \(R_0,R_1,R_2\).  This is an exact
Taylor--Hadamard decomposition.  The pure-\(t\) Taylor formula uses the
contact conditions to remove the powers \(t,\ldots,t^{\nu-1}\), while
Hadamard division in \(I\) gives
\[
  h_p(t,I)-h_p(t,0)=I g(t,I,p),
  \qquad g(0,0,p)=1,
\]
and a second division writes \(g-1=tR_1+IR_2\).  The normalization of the
first integral fixes the linear \(I\)-term.

Writing \(W_p=a_p\partial_t+b_p\partial_I\), the plus field satisfies
\begin{equation}
  \dot t=1+h_pa_p(t,I),
  \qquad
  \dot I=h_pb_p(t,I),
  \qquad
  b_p(0,0)=\beta(p).
  \label{eq:grazing-plus-field}
\end{equation}
For \(\rho=r^\nu\), put \(t=ru\) and
\(I=r^\nu+r^{\nu+1}w\).  Then
\[
  H(u,w,r,p)
  =r^{-\nu}h_p(ru,r^\nu+r^{\nu+1}w)
  =1-\frac{\kappa_\nu(p)}{\nu!}u^\nu+r\widetilde H(u,w,r,p)
\]
extends jointly smoothly to \(r=0\) on bounded sets.  On the plus
excursion,
\begin{equation}
  \frac{dw}{du}
  =\frac{H(u,w,r,p)b_p(ru,r^\nu+r^{\nu+1}w)}
  {1+r^\nu H(u,w,r,p)a_p(ru,r^\nu+r^{\nu+1}w)}.
  \label{eq:grazing-rescaled-ode}
\end{equation}
The denominator is uniformly positive for small \(r\).

We first localize every seam event in the entire fixed grazing box.  Choose
the box as \(|t|\leq\delta\).  Shrinking \(\delta\) and the parameter set,
and then choosing \(M>1\), gives \(c_0,r_0>0\) such that
\begin{equation}
  h_p(t,I)\leq-c_0|t|^\nu
  \quad\text{if}\quad
  Mr\leq|t|\leq\delta,
  \quad |I-r^\nu|\leq r^\nu,
  \quad0<r<r_0.
  \label{eq:grazing-annular-sign}
\end{equation}
Indeed, on this annulus \(r^\nu\leq M^{-\nu}|t|^\nu\); after \(M\) is
large, the terms \(tI R_1\) and \(I^2R_2\) are absorbed by the leading
term.

This estimate closes a continuation argument.  As long as
\(|I-r^\nu|\leq r^\nu\), the active part is confined to \(|t|<Mr\).
There \(|h_p|=O(r^\nu)\), \(\dot t\) is uniformly positive, and
\eqref{eq:grazing-plus-field} gives \(|dI/dt|\leq Cr^\nu\).  The active
interval has length at most \(2Mr\), so
\begin{equation}
  I(t)=r^\nu+O(r^{\nu+1})
  \label{eq:grazing-fixed-box-I-control}
\end{equation}
throughout the fixed box.  For small \(r\) this strictly improves the
bootstrap strip.  Hence all seam events lie in \(|t|<Mr\), and the field is
inactive on the rest of the box.

Set
\[
  a_\nu(p)=\left(\frac{\nu!}{\kappa_\nu(p)}\right)^{1/\nu}.
\]
On \([-M,M]\), \(H(u,0,r,p)\) converges in \(C^1\), uniformly in \(p\),
to
\[
  H_0(u,p)=1-\frac{u^\nu}{a_\nu(p)^\nu}.
\]
Uniform neighbourhoods of the two simple roots
\(u=\pm a_\nu(p)\), together with the sign gap on their complement, give
a unique first entry \(u_-(r,p)\) near \(-a_\nu(p)\).

Let \(w(u;r,p)\) solve \eqref{eq:grazing-rescaled-ode} with
\(w(u_-(r,p);r,p)=0\).  Smooth ODE dependence bounds \(w\) and
\(\partial_uw\) uniformly.  The physical event function
\[
  E(u,r,p)=H(u,w(u;r,p),r,p)
\]
therefore converges to \(H_0\) in \(C^1\).  A second implicit-function
argument gives the unique first exit \(u_+(r,p)\) near \(a_\nu(p)\), and
the sign gap proves \(E>0\) exactly between the two roots.  After exit the
physical trajectory has constant \(w_+(r,p)\); applying the same sign gap
to \(H(u,w_+,r,p)\), together with \eqref{eq:grazing-annular-sign}, shows
that no further seam event occurs in the fixed box.

The event maps and
\[
  \mathcal A_\nu(r,p)=w(u_+(r,p);r,p)
\]
are jointly smooth.  Hence
\(\mathcal D_{\mathrm{gr}}(r^\nu,p)=r^{\nu+1}\mathcal A_\nu(r,p)\).
At \(r=0\), equation \eqref{eq:grazing-rescaled-ode} gives
\begin{align*}
  \mathcal A_\nu(0,p)
  &=\beta(p)\int_{-a_\nu(p)}^{a_\nu(p)}
    \left(1-\frac{u^\nu}{a_\nu(p)^\nu}\right)\dd u\\
  &=\frac{2\nu}{\nu+1}a_\nu(p)\beta(p)
  =\mathfrak g_\nu(p).
\end{align*}

For \(\rho<0\), shrinking the box keeps the reference arc in \(h_p<0\),
so the physical and reference \(X_p^-\)-excursions agree.  At \(\rho=0\), the
physical field \(X_p^-+(h_p)_+W_p\) is locally Lipschitz.  The central
reference \(X_p^-\)-arc is a solution and touches the seam only at \(z_p\), so
uniqueness makes it the physical arc as well.  Before entry the two
trajectories coincide; after their exits both follow \(X_p^-\), and \(I_p\)
is its exact first integral.  Thus the computed \(I_p\)-difference reaches
the fixed outgoing section.  Finally, one-sided Hadamard factorization
\(\mathcal A_\nu(r,p)-\mathcal A_\nu(0,p)=rB_\nu(r,p)\) yields the full
smooth expansion in the statement.
\end{proof}

\subsection{Fixed and weighted-value transmission}

\begin{proof}[Proof of
\cref{prop:reference-return-composition,prop:weighted-value-transmission}]
We prove both transmission propositions together.  First write \(q=s^\nu\) and
\(\mathcal D_{\mathrm{gr}}(q,p)=s^{\nu+1}\mathcal A_\nu(s,p)\).  The exact
divided difference is
\begin{equation}
  \begin{aligned}
  &\mathcal R_p(q+\mathcal D_{\mathrm{gr}}(q,p))-\mathcal R_p(q)\\
  &\quad={}
  s^{\nu+1}\mathcal A_\nu(s,p)
  \int_0^1(\mathcal R_p)'
  \left(q+\theta s^{\nu+1}\mathcal A_\nu(s,p)\right)\dd\theta .
  \end{aligned}
  \label{eq:regular-transmission-divided-difference}
\end{equation}
The factor after \(s^{\nu+1}\) is jointly smooth for independent
\(s\geq0,q,p\), and its value at the origin is
\(L_0\mathfrak g_\nu(0)\).  This proves the full ramified expansion,
including its leading coefficient.  Under the common-itinerary,
transverse-section, and no-earlier-intersection hypotheses of
\cref{prop:reference-return-composition}, the displayed composition is the
actual first return.  This operation postcomposes the local passage
correction with the regular transition; a change of Poincar\'e section
instead acts by conjugacy.

For the weighted-value proposition, let \(H_r\) denote the active correction
along the stated weighted path.  The same divided difference writes the
transmitted correction as
\[
  H_r\int_0^1\mathcal R_{p_r}'(q_r+\theta H_r)\dd\theta.
\]
Uniformly on the compact normalized set, the integral is
\(L_0+O_K(r^2)\): the largest transverse parameter is \(O(r^2)\), while
\(q_r=O(r^\nu)\) and \(H_r=O(r^{\nu+1})\).  Multiplying the assumed
\(H_r=\beta_0 r^{\nu+1}\mathcal C+O_K(r^{\nu+2})\) proves the transmitted
estimate with the same \(O_K(r^{\nu+2})\) error.
\end{proof}

\subsection{Preparation and the transverse value estimate}

\begin{proof}[Proof of \cref{lem:transverse-cap-preparation}]
The regularity of the seam gives its graph by the implicit function
theorem.  Before centering, solve
\[
  \partial_t^{\nu-1}\Phi(\tau(p),p)=0
\]
for \(\tau(p)\).  The derivative with respect to \(\tau\) at the central
contact is \(\partial_t^\nu\Phi(0,0)=\nu!a\neq0\).  The centered graph is
\[
  \Phi_{\mathrm{new}}(t,p)
  =\Phi_{\mathrm{old}}(t+\tau(p),p)-\Phi_{\mathrm{old}}(\tau(p),p).
\]
Thus translations of \(t\) and \(I\) remove the \(t^{\nu-1}\) and
constant coefficients while preserving \(X_p^-=\partial_t\).  The rank
hypothesis and constant-rank theorem give a transverse slice on which the
remaining coefficients are parameter coordinates.  Taylor's formula gives
\eqref{eq:prepared-cap-graph}.  Hadamard division by the regular graph gives
\eqref{eq:prepared-cap-switching-function}, and the chosen side orientation
makes \(g>0\) after shrinking.
The lower contact monomials disappear after \(\nu\) derivatives, whereas
\[
  \partial_t^\nu\bigl(t^{\nu+1}R(t,\mathbf u)\bigr)=O(t)
\]
uniformly for small \(\mathbf u\).  Since
\(\nu!a(\mathbf u)\to\nu!a>0\), one further shrink gives
\eqref{eq:prepared-cap-finite-type}.
\end{proof}

\begin{proof}[Proof of \cref{lem:cap-weighted-homogeneity}]
The identity follows from \(t=rs\).  On a compact coefficient set, the
leading term \(-at^\nu\) places all positive sets in a common bounded
interval.  Dominated convergence, or the elementary inequality
\begin{equation}
  |x_+-y_+|\leq|x-y|,
  \label{eq:positive-part-lipschitz}
\end{equation}
proves continuity, including at multiple boundary roots and at component
births and mergers.
\end{proof}

\begin{proof}[Proof of \cref{thm:transverse-cap-blow-up}]
Write
\[
  \mathbf u_r=(r^{\nu-1}u_1,\ldots,r^2u_{\nu-2})
\]
and let \(I_r(t)\) be the physical trajectory launched from
\(I_r(-T)=r^\nu\rho\).  Since the normalized coefficients range over a
compact set \(K\), there is \(M>1\), independent of \(r\) and the chosen
point in \(K\), for which
\begin{equation}
  \Phi_{\mathbf u_r}(t)>I
  \quad\text{whenever}\quad
  Mr\leq|t|\leq T,
  \qquad |I-r^\nu\rho|\leq r^\nu,
  \label{eq:cap-support-localization}
\end{equation}
after reducing \(r_K\).  Indeed, \(a(\mathbf u_r)\geq a/2\), and for
\(|t|\geq Mr\),
\[
  \left|\sum_{j=1}^{\nu-2}r^{\nu-j}u_jt^j\right|
  \leq |t|^\nu
  \sum_{j=1}^{\nu-2}|u_j|M^{-(\nu-j)}.
\]
Choose \(M\) so the sum is a small fraction of \(a|t|^\nu\); shrinking the
fixed box absorbs \(t^{\nu+1}R\), and increasing \(M\) handles the bounded
level \(\rho\).  The same estimates give \(|h_{\mathbf u_r}|\leq Cr^\nu\)
for \(|t|\leq Mr\) in the bootstrap strip.

On the active side, \(t\) is a valid independent variable and
\begin{equation}
  \frac{\dd I_r}{\dd t}
  =\frac{(h_{\mathbf u_r})_+B_{\mathbf u_r}(t,I_r)}
  {1+(h_{\mathbf u_r})_+A_{\mathbf u_r}(t,I_r)}.
  \label{eq:physical-cap-I-equation}
\end{equation}
The denominator is at least \(1/2\).  Before a hypothetical first escape
from \(|I_r-r^\nu\rho|<r^\nu\), the right-hand side is supported on an
interval of length at most \(2Mr\) and has size \(O(r^\nu)\).  Hence
\begin{equation}
  I_r(t)=r^\nu\rho+O_K(r^{\nu+1}),
  \label{eq:physical-reference-cap-error}
\end{equation}
which closes the continuation argument and confines every contributing
component to the rescaled window.

For \(|s|\leq M\), set
\(I_r(rs)=r^\nu\rho+r^{\nu+1}w_r(s)\).  The preceding estimate bounds
\(w_r\).  Taylor expansion of the prepared graph gives, uniformly on the
same compact set,
\begin{align}
  &r^{-\nu}h_{\mathbf u_r}
  (rs,r^\nu\rho+r^{\nu+1}w_r(s))
  \notag\\
  &\qquad={}
  g(0,0,0)\left(\rho-as^\nu-
    \sum_{j=1}^{\nu-2}u_js^j\right)+O_K(r).
  \label{eq:rescaled-physical-cap-function}
\end{align}
Here \(a(\mathbf u_r)-a=O(r^2)\), while both the graph remainder and the
physical displacement contribute \(O(r)\).  Applying
\eqref{eq:positive-part-lipschitz} in \eqref{eq:physical-cap-I-equation},
and using smoothness of \(A_{\mathbf u},B_{\mathbf u}\), yields
\begin{equation}
  w_r'(s)=\beta\left(\rho-as^\nu-
  \sum_{j=1}^{\nu-2}u_js^j\right)_++O_K(r)
  \label{eq:rescaled-physical-cap-ode}
\end{equation}
uniformly for \(|s|\leq M\).  This uniform argument treats the full positive
set at once, independently of its number of connected components.

The trajectory agrees with the reference before \(-Mr\), so
\(w_r(-M)=0\), and its \(I\)-coordinate is constant after \(Mr\).
Integration gives
\[
  r^{-(\nu+1)}\mathcal D_{\mathrm{cap}}
  (r^\nu\rho,\mathbf u_r)
  =\beta\mathcal C_{\nu,a}(\rho,\mathbf u)+O_K(r).
\]
The definition of \(M\) makes the cap integrand vanish outside
\([-M,M]\).  Multiplication by \(r^{\nu+1}\) proves the claimed estimate,
uniformly across the polynomial discriminant.
\end{proof}

\section{Boundary-event sensitivity}
\label{app:boundary-events}

This appendix proves
\cref{thm:prepared-multiwell-boundary-curvature}.  Differentiating a
regular passage creates one endpoint term at every entry and exit.  They
all have the same sign and dominate the bounded interior variation.  The
separate issue at births and mergers is first-order regularity: we prove it
directly from scalar-flow difference quotients and then join the regular
chambers.  Standard transverse-event sensitivity gives the first-order
background
\cite{BurdenSastryKoditschekRevzen2016,KhanBarton2017,
SacconVandeWouwNijmeijer2014}; here the event slopes collapse, so the
multi-event sign, uniform divergence, and continuation across the event
discriminant require a separate calculation.

\begin{proof}[Proof of
  \cref{thm:prepared-multiwell-boundary-curvature}]
Choose \(T\) so that \(\Phi_0(t)>0\) on
\([-T,T]\setminus\{0\}\).  The right-hand side of
\eqref{eq:multiwell-physical-ode} is locally Lipschitz in \(I\), uniformly
in \((t,\theta)\), so its solution is unique and depends continuously on
\((q,\theta)\).  The central solution is \(I\equiv0\).  Since
\(\Phi_0(\pm T)>0\), shrinking the initial and parameter germs makes both
endpoints strictly inactive.  If \(q\leq\tau_\theta\), the constant
function \(I\equiv q\) is the unique solution, and hence \(D=0\).

\paragraph{Finite event topology.}
The finite-type inequality implies, by repeated Rolle, that
\(\Phi_\theta'\) has at most \(\nu-1\) distinct zeros.  Between consecutive
critical times the graph is strictly monotone.  At any boundary point of
the active set, with \(h=I-\Phi_\theta\),
\begin{equation}
  h'=-\Phi_\theta'
  \qquad\text{when }h=0.
  \label{eq:multiwell-boundary-direction}
\end{equation}
Thus a regular entry has \(\Phi_\theta'<0\), a regular exit has
\(\Phi_\theta'>0\), and one monotonicity interval contains at most one
event.  The active set is consequently a finite union of intervals, paired
because both box endpoints are inactive.  There are at most \(\nu\)
monotonicity intervals and therefore at most \(\nu\) events.  Since the
events are paired, the number of active components satisfies
\(N\leq\nu/2\), proving \eqref{eq:multiwell-component-bound}.  The active
set cannot contain a seam
interval: if \(h=0\) on an interval, then \(I'=0\) there and hence
\(\Phi_\theta'=0\) there, contradicting
\(\partial_t^\nu\Phi_\theta\geq\kappa_*>0\).

A nontransverse event can occur only at a critical time \(t_c\) of
\(\Phi_\theta\).  Uniqueness for both the forward and reversed scalar
equations makes the time-\(t_c\) flow map
\(q\mapsto I(t_c;q,\theta)\) injective and strictly increasing.  Therefore
\(I(t_c;q,\theta)=\Phi_\theta(t_c)\) has at most one solution in \(q\).
There are at most \(\nu-1\) exceptional levels in
\([\tau_\theta,\rho_0)\), including the threshold when it is exceptional.

\paragraph{First variation through births and mergers.}
Put
\[
  f_\theta(t,I)=(I-\Phi_\theta(t))_+B(t,I,\theta).
\]
For two initial values \(q\) and \(q+\delta\), strict order permits the
bounded divided difference
\[
  A_\delta(t)=
  \frac{f_\theta(t,I(t;q+\delta,\theta))
       -f_\theta(t,I(t;q,\theta))}
       {I(t;q+\delta,\theta)-I(t;q,\theta)}.
\]
The solution difference satisfies a scalar linear equation, so
\begin{equation}
  \frac{I(t;q+\delta,\theta)-I(t;q,\theta)}{\delta}
  =\exp\left(\int_{-T}^{t}A_\delta(s)\dd s\right).
  \label{eq:multiwell-difference-quotient}
\end{equation}
The zero set of \(h(t)=I(t;q,\theta)-\Phi_\theta(t)\) is finite by the
event argument above; the exclusion of seam intervals also rules out
accumulation.  Away from that null set, dominated
convergence gives
\begin{equation}
  J(t):=\partial_qI(t;q,\theta)
  =\exp\left(\int_{-T}^{t}
    \mathbf 1_{\{h>0\}}
    \mathcal F_I(s,I(s),\theta)\dd s\right)>0,
  \label{eq:multiwell-first-variation}
\end{equation}
where \(\mathcal F=(I-\Phi_\theta)B\) on the active set.  The same
argument for a convergent sequence \((q_j,\theta_j)\to(q,\theta)\) uses
uniform convergence of the trajectories and almost-everywhere convergence
of the active indicators.  It proves joint continuity of \(J\), including
at births, mergers, and simultaneous equal-depth events.  Hence
\begin{equation}
  D_q(q,\theta)=J(T;q,\theta)-1
  \label{eq:multiwell-outgoing-first-variation}
\end{equation}
exists and is jointly continuous.  At \(q=\tau_\theta\) the active set
shrinks to a null set, so \(J(T)\to1\); this matches the inactive
derivative.  Transverse event selection and \(C^2\) ODE dependence give
\(C^2\) dependence on every regular chamber.

\paragraph{The multi-event second variation.}
On a regular chamber use
\eqref{eq:multiwell-active-union}.  Equation
\eqref{eq:multiwell-first-variation} gives
\[
  \log J(T)=\int_{\mathcal A(q,\theta)}
  \mathcal F_I(t,I(t),\theta)\dd t.
\]
At any regular event \(c\), differentiating
\(I(c;q,\theta)=\Phi_\theta(c)\) and using \(I'(c)=0\) yields
\begin{equation}
  \partial_qc=\frac{J(c)}{\Phi_\theta'(c)}.
  \label{eq:multiwell-event-derivative}
\end{equation}
Differentiate the integral over all active components.  Since
\(\mathcal F_I=B\) at every event, the endpoint contributions from
\eqref{eq:multiwell-event-derivative}, together with the interior
variation \(\mathcal F_{II}J\), give
\eqref{eq:physical-passage-curvature-formula}.  Notice that the factor
\(J(a_k)\) remains at every later entry; only the first entry has
\(J(a_1)=1\).

\paragraph{Uniform endpoint domination.}
As \(|q|+\|\theta\|\to0\), the trajectories converge uniformly to zero.
If active times failed to converge uniformly to \(0\), a limiting active
time \(t_*\neq0\) would satisfy \(\Phi_0(t_*)\leq0\), contrary to the
choice of \(T\).  Consequently the total active length \(L(\rho)\) and
largest regular event slope \(\omega(\rho)\) defined in
\eqref{eq:multiwell-active-length-modulus}--%
\eqref{eq:multiwell-event-slope-modulus} both tend to zero with \(\rho\).
The empty-event convention makes \(\omega\) defined throughout the germ.
Every regular event lies in the closure
of its active set, so the same localization applies to its position and
slope.

Shrink once more so that \(B\) has the sign of \(\beta\) and
\(|B|\geq b_0>0\).  Choose \(A,C_2\) with
\(|\mathcal F_I|\leq A\) and \(|\mathcal F_{II}|\leq C_2\) on the
common box.  From \eqref{eq:multiwell-first-variation},
\[
  e^{-AL(\rho)}\leq J(t)\leq e^{AL(\rho)},
\]
so \(1/2\leq J\leq2\) for small \(\rho\).  Every nonempty active set has
at least one entry--exit pair.  Multiplication of the bracket in
\eqref{eq:physical-passage-curvature-formula} by \(\sgn\beta\) therefore
gives the lower bound
\[
  \frac{b_0}{\omega(\rho)}-2C_2L(\rho).
\]
Since \(\omega(\rho)\to0\) and \(L(\rho)\to0\), shrink \(\rho_1\) so that
this lower bound is nonnegative for \(0<\rho\leq\rho_1\).  All additional
wells add positive endpoint terms.  The outer factor
\(J(T)\geq1/2\) proves the explicit bound
\eqref{eq:physical-passage-quantitative-curvature}.  The same limits show
that this bound yields
\eqref{eq:physical-passage-curvature-sign} with any prescribed modulus.
Equations \eqref{eq:multiwell-first-variation} and
\eqref{eq:multiwell-outgoing-first-variation} also give
\(D,D_q\to0\).

For fixed \(\theta\), partition
\([q_1,q_2]\) at its finitely many exceptional levels.  Integrating the
regular curvature estimate first on compact subintervals trimmed away from
the exceptional endpoints, and then using continuity of \(D_q\) as the
trimming tends to zero, gives
\eqref{eq:multiwell-derivative-increment}.  This completes the proof.
\end{proof}

\subsection{Transmission of the boundary curvature}

\begin{proof}[Proof of
  \cref{prop:threshold-curvature-transmission,%
cor:multiwell-curvature-transmission}]
Set
\[
  G(q,p)=\mathcal T_{\mathcal R}[H](q,p)
  =\mathcal R_p(q+H(q,p))-\mathcal R_p(q).
\]
The assumed \(C^1\) regularity of \(H\) makes \(G_q\) continuous across
the active-level interval.  At every regular point,
\[
  G_{qq}
  =\mathcal R_p''(q+H)(1+H_q)^2
  +\mathcal R_p'(q+H)H_{qq}
  -\mathcal R_p''(q).
\]
Along the active germ, continuity of \(\tau\) and the assumptions
\(H,H_q\to0\) give \((q+H,p)\to(0,0)\).  After shrinking,
\(\mathcal R_p'(q+H)\) has the sign of \(L_0\) and
modulus at least \(|L_0|/2\).  The other second-derivative terms remain
bounded, while \(H_{qq}\) has the sign of \(\beta_0\) with divergent
modulus.  Hence \(G_{qq}\) has the sign of \(L_0\beta_0\) and any
prescribed lower modulus on every regular interval.  Partition at the
finite exceptional set and use continuity of \(G_q\) to join these
interval estimates.  This proves
\cref{prop:threshold-curvature-transmission}.

For \cref{cor:multiwell-curvature-transmission}, specialize
\((H,p,\beta_0)=(D,\theta,\beta)\).  The preceding boundary-event theorem
supplies the active-side hypotheses.  On the inactive side \(D=D_q=0\),
so the transmitted correction and its first derivative vanish there and
at the threshold, as claimed.
\end{proof}

\section{Closed realization: return and rank calculations}
\label{app:closed-realization}

This appendix supplies the common-return-domain, Poincar\'e expansion,
coefficient-rank, and contact-jet-rank calculations used in
\cref{sec:closed-polynomial-realization}.

\begin{proof}[Proof of \cref{thm:closed-polynomial-realization}]
Work in polar coordinates on an annulus about the unit circle.  Any vector
field \(R+(H/2)V\) satisfies
\begin{equation}
  \dot\theta=1,
  \qquad
  \dot\varrho=(1+\varrho)H.
  \label{eq:closed-polynomial-polar-equations}
\end{equation}
Indeed,
\(x\dot y-y\dot x=x^2+y^2\), while
\(\dot\varrho=2x\dot x+2y\dot y=(x^2+y^2)H\).
For the physical field, \(H=G_\lambda+\beta(h_\ell)_+\).
To make the return domain explicit, choose a closed annulus
\(\lvert\varrho\rvert\leq a\), with \(a<1\).  The radial right-hand side is
locally Lipschitz in \(\varrho\), uniformly in the angular variable, and it
vanishes at \((\varrho,\lambda)=(0,0)\).  Hence on this annulus
\[
  |\dot\varrho|\leq L|\varrho|+C\lVert\lambda\rVert
\]
for fixed \(L,C\).  Gronwall's inequality supplies a smaller interval
\(\lvert\varrho\rvert<a_0\) and a parameter neighbourhood for which every
such solution remains in \(\lvert\varrho\rvert<a\) throughout
\(0\leq\theta-\theta_0\leq2\pi\).  Standard continuation then gives a
common full-turn domain.  Since
\(\theta(t)=\theta_0+t\), the selected ray cannot be met for
\(0<t<2\pi\); its time-\(2\pi\) intersection is therefore the first
return.  At
\(\lambda=0\), equation \eqref{eq:closed-polynomial-polar-equations} on
\(\varrho=0\) gives \(\dot\varrho=0\), so the unit circle is a
\(2\pi\)-periodic orbit.  Along it,
\begin{equation}
  h_\ell(\cos t,\sin t)
  =-(1-\cos t)^\ell
  =-2^{-\ell}t^{2\ell}+O(t^{2\ell+2}).
  \label{eq:closed-polynomial-osculation}
\end{equation}
It is strictly negative away from \(t=0\pmod{2\pi}\).  Thus \(z\) is the
unique grazing point on the central circle and
\(- (X^-_0)^\nu h_\ell(z)=\nu!/2^\ell\).

We also verify that the unfolding stays in the fixed-contact-order stratum.
Parametrize by \(t=\theta\) the \(X^-_\lambda\)-orbit through \(z\), and
write its radial variable as \(\varrho_\lambda(t)\), with
\(\varrho_\lambda(0)=0\).  Its equation is
\begin{equation}
  \varrho_\lambda'
  =(1+\varrho_\lambda)
  \left[
    \gamma \varrho_\lambda^n
    +\sum_{j=1}^{n-1}\lambda_j\varrho_\lambda^j
    +\lambda_0
      \bigl(1-\sqrt{1+\varrho_\lambda}\cos t\bigr)^\ell
  \right].
  \label{eq:closed-polynomial-tangent-orbit}
\end{equation}
Since the inhomogeneous term at \(\varrho=0\) is \(O(t^{2\ell})\), smooth ODE
estimates (or variation of constants followed by Gronwall's inequality)
give \(\varrho_\lambda(t)=O(t^{2\ell+1})\), uniformly for small \(\lambda\).
It follows that
\[
  1-\sqrt{1+\varrho_\lambda(t)}\cos t
  =\frac12t^2+O(t^4)+O(t^{2\ell+1}),
\]
and hence
\begin{equation}
  h_\ell\left(
    \sqrt{1+\varrho_\lambda(t)}\cos t,
    \sqrt{1+\varrho_\lambda(t)}\sin t
  \right)
  =-2^{-\ell}t^{2\ell}+O(t^{2\ell+1}).
  \label{eq:closed-polynomial-persistent-contact}
\end{equation}
This proves persistence of the order and of \(\kappa_\nu\).

Hadamard factorization is already explicit in
\eqref{eq:closed-polynomial-branches}:
\[
  \frac{X^+_\lambda-X^-_\lambda}{h_\ell}
  =\frac\beta2V.
\]
For \(\ell\geq2\), one has \(dh_\ell(z)=2\,dx\), and therefore
\begin{equation}
  dh_\ell(z)\left[\frac\beta2V(z)\right]=\beta.
  \label{eq:closed-polynomial-mismatch}
\end{equation}

Choose a ray \(\theta=-\delta\), with \(\delta>0\) small, immediately
before the grazing box.  Let \(\sigma(\lambda)\) be the \(\varrho\)-coordinate
on this ray of the reference \(X^-_\lambda\)-orbit that reaches \(z\).  Smooth ODE
dependence makes \(\sigma\) smooth.  We use
\(q=\varrho-\sigma(\lambda)\) on the ray.  Let \(I_\lambda\) be the first
integral normalized by \eqref{eq:grazing-first-integral}, transported to
this ray along the reference \(X^-_\lambda\)-flow, and write its restriction as
\(\rho_\lambda(q)\).  It vanishes at \(q=0\).  On the central orbit the
radial variational multiplier from the ray to \(z\) is one, while
\[
  dh_\ell(z)=2\,dx,
  \qquad
  \left.\frac{\partial}{\partial\varrho}
    \bigl(\sqrt{1+\varrho},0\bigr)\right|_{\varrho=0}
  =\left(\frac12,0\right).
\]
The normalization \(dI_0(z)=dh_\ell(z)\) therefore gives
\(\partial_q\rho_0(0)=1\).  Hadamard factorization and smooth parameter
dependence now yield the exact relation
\begin{equation}
  \rho_\lambda(q)=q\,u(q,\lambda),
  \qquad
  u(0,0)=1,
  \qquad
  u(q,\lambda)>0
  \label{eq:closed-polynomial-input-factor}
\end{equation}
after shrinking.

The compact part of the central circle outside the selected grazing box has
a strictly negative maximum of \(h_\ell\).  Continuous dependence, together
with the common annulus above, therefore keeps all relevant nearby physical
and reference arcs in \(h_\ell<0\) there.  Hence the selected box contains the
only seam events in the return itinerary.  Substitution of
\eqref{eq:closed-polynomial-input-factor} into
\cref{prop:even-order-grazing-comparison} preserves the ramified class:
\[
  \rho_{\lambda,+}^{1+1/\nu}
  =
  q_+^{1+1/\nu}u(q,\lambda)^{1+1/\nu},
  \qquad
  \rho_{\lambda,+}^{1/\nu}
  =
  q_+^{1/\nu}u(q,\lambda)^{1/\nu}.
\]
Both factors are smooth in the stipulated variables because \(u>0\).
The outgoing coordinate change and the remaining reference \(X^-_\lambda\)-flow are smooth
local diffeomorphisms.  Their divided-difference composition, exactly as in
\cref{prop:reference-return-composition}, therefore gives
\eqref{eq:closed-polynomial-displacement} with its single ramified term.

On the central reference \(X^-_0\)-orbit the transverse variational equation along
\(\varrho=0\) has multiplier one, because
\[
  \left.\partial_\varrho
    [\gamma(1+\varrho)\varrho^n]\right|_{\varrho=0}=0.
\]
Thus the outgoing reference derivative is one.  Together with
\(u(0,0)=1\), the total
transmission coefficient in the \(q\)-coordinate is \(L_0=1\).  Inserting
\eqref{eq:closed-polynomial-contact-data} into
\eqref{eq:grazing-leading-coefficient} gives
\[
  K(0,0,0)
  =\frac{2\nu}{\nu+1}
    \left(2^\ell\right)^{1/\nu}\beta
  =\frac{2\nu}{\nu+1}\sqrt2\,\beta.
\]

It remains to calculate the smooth jet.  Let \(\Phi_\lambda\) denote the
\(\varrho\)-transition of the reference \(X^-_\lambda\)-flow over one full
turn from the selected ray.  In the translated coordinate chosen above,
its smooth reference displacement is
\begin{equation}
  S(q,\lambda)
  =
  \Phi_\lambda\bigl(\sigma(\lambda)+q\bigr)
  -\sigma(\lambda)-q.
  \label{eq:closed-polynomial-reference-return}
\end{equation}
At \(\lambda=0\), the reference radial equation is the autonomous equation
\[
  \varrho'=\gamma(1+\varrho)\varrho^n.
\]
Its time-\(2\pi\) map is
\(q\mapsto q+2\pi\gamma q^n+O(q^{n+1})\), which proves the expansion for
\(S(q,0)\) in \eqref{eq:closed-polynomial-leading-data}.  Denote by
\(U_j(q)\) the
first variation of the smooth reference displacement in the
\(\lambda_j\)-direction.  For \(1\leq j<n\), the tangent orbit remains
\(\varrho=0\), so \(\partial_{\lambda_j}\sigma(0)=0\).  Write
\(\Phi^0_{b,a}\) for the central evolution from angle \(a\) to angle
\(b\).  If \(\varrho_q(\theta)\) is the central reference solution launched
from \(q\), the ordinary variational formula integrates
\[
  D\Phi^0_{\theta_0+2\pi,\theta}
    \bigl(\varrho_q(\theta)\bigr)
  (1+\varrho_q(\theta))\varrho_q(\theta)^j
\]
from \(\theta_0=-\delta\) to \(\theta_0+2\pi\).  Since
\(\varrho_q(\theta)=q+O(q^n)\) and the variational multiplier is
\(1+O(q^{n-1})\), this gives
\begin{equation}
  U_j(q)=2\pi q^j+O(q^{j+1}).
  \label{eq:closed-polynomial-positive-jet-variations}
\end{equation}
For \(j=0\), differentiating
\eqref{eq:closed-polynomial-reference-return} at \(q=0\) gives
\[
  U_0(0)
  =
  \left.\partial_{\lambda_0}\Phi_\lambda(0)\right|_{\lambda=0}
  +\bigl(\Phi_0'(0)-1\bigr)\partial_{\lambda_0}\sigma(0).
\]
The second term vanishes because \(\Phi_0'(0)=1\), while the first is the
full-period forcing.  Hence
\begin{equation}
  U_0(0)
  =\int_0^{2\pi}(1-\cos\theta)^\ell\dd\theta=C_\ell.
  \label{eq:closed-polynomial-constant-variation}
\end{equation}
Finally, \(1-\cos\theta=2\sin^2(\theta/2)\) and the standard even-power
integral give the closed expression for \(C_\ell\) in
\eqref{eq:closed-polynomial-rank-diagonal}.  Equations
\eqref{eq:closed-polynomial-positive-jet-variations} and
\eqref{eq:closed-polynomial-constant-variation} make the derivative of
the coefficient map triangular with determinant
\begin{equation}
  C_\ell(2\pi)^{n-1}\neq0.
  \label{eq:closed-polynomial-rank-determinant}
\end{equation}

Thus \(c=2\pi\gamma\),
\(d=2\nu\sqrt2\,\beta/(\nu+1)\), and the coefficient rank is \(n\).
Apply \cref{thm:sharp-cyclicity,cor:stability}.  Since every member of
the unfolding has an actual time-\(2\pi\) return, the fixed points attaining
the scalar bounds correspond precisely to actual periodic orbits of
\eqref{eq:closed-polynomial-physical-field}.  Their simplicity, hyperbolicity
after shrinking, and alternating stability follow from the same results.
\end{proof}

\begin{proof}[Proof of \cref{thm:closed-contact-jet-transverse-extension}]
The two extensions in \eqref{eq:closed-transverse-field} agree on
\(h_{\ell,\eta}=0\), so the physical field is continuous.  Because
\(s\mapsto s_+\) is Lipschitz and \(h_{\ell,\eta}\) is smooth, the physical
field is locally Lipschitz in the state variables.  In polar coordinates
its equations are
\[
  \dot\theta=1,
  \qquad
  \dot\varrho=(1+\varrho)
    \left[G_\lambda+\beta(h_{\ell,\eta})_+\right].
\]
On a fixed compact annulus the estimate used in the proof of
\cref{thm:closed-polynomial-realization} becomes
\[
  |\dot\varrho|
  \leq L|\varrho|+C(\lVert\lambda\rVert+\lVert\eta\rVert).
\]
After shrinking the state and parameter domains, Gronwall's inequality
gives a common full turn.  The identity \(\dot\theta=1\) proves the first
claim.

For the contact rank, order the rows of
\(\mathcal J_\nu^{\mathrm{cl}}\) by Lie-derivative
order \(0,1,\ldots,\nu-1\).  Along the central orbit
\((x,y)=(\cos t,\sin t)\),
\begin{equation}
  \left.
  \frac{\partial}{\partial\eta_j}
  (X^-_0)^k h_{\ell,\eta}(z)
  \right|_{\eta=0}
  =\left.
    \frac{\dd^k}{\dd t^k}\sin^j t
   \right|_{t=0}.
  \label{eq:closed-contact-triangular-entry}
\end{equation}
The right-hand side vanishes for \(k<j\) and equals \(j!\) for \(k=j\).
Choose a normal vector \(N\) with \(dh_\ell(z)[N]=1\).  Its column has
first entry one.  The column corresponding to displacement along
\(X^-_0(z)\) has entries
\((X^-_0)^{k+1}h_\ell(z)\), hence vanishes through row \(\nu-2\) and has
last entry
\[
  (X^-_0)^\nu h_\ell(z)
  =-\kappa_\nu=-\frac{\nu!}{2^\ell}\neq0.
\]
After eliminating entries below the indicated pivots, the
matrix with columns
\(N,\eta_1,\ldots,\eta_{\nu-2},X^-_0(z)\) is triangular, with determinant
\begin{equation}
  -\frac{\nu!}{2^\ell}
   \prod_{j=1}^{\nu-2}j!\neq0
  \label{eq:closed-contact-rank-determinant}
\end{equation}
up to the orientation chosen for the columns.  The implicit function
theorem therefore makes the order-\(\nu\) contact locus a codimension
\(\nu-2\) parameter submanifold.  By
\cref{thm:closed-polynomial-realization}, \((z,\lambda,0)\) already belongs
to that locus for every small \(\lambda\).  Local uniqueness in the
implicit function theorem identifies the locus with \(\eta=0\).

Finally, the reference field \(X^-_\lambda\) is independent of the
\(\eta\)-parameters, whereas the derivative of
\(\lambda\mapsto(a_0,\ldots,a_{n-1})\) is invertible by
\eqref{eq:closed-polynomial-rank-determinant}.  The derivative of
\eqref{eq:closed-total-codimension-map} is therefore block diagonal, with
invertible blocks of sizes \(n\) and \(\nu-2\).  This proves the total-rank
and codimension claims.

Apply \cref{lem:transverse-cap-preparation} in the \(X^-_\lambda\)-flow
box at the selected contact.  The transverse rank just proved supplies the
prepared contact coordinates, and
\eqref{eq:prepared-cap-finite-type} supplies a uniform finite-type
constant.  In \eqref{eq:closed-transverse-field} the branch difference is
\((\beta/2)h_{\ell,\eta}V\), and at the base point
\[
  dh_{\ell,0}(z)\left[\frac{\beta}{2}V(z)\right]=\beta\neq0.
\]
Continuity therefore keeps the active coefficient fixed in sign on the
common box.

It remains to connect this local passage to the complete return uniformly
in \(\eta\).  On the compact complement of the selected grazing box, the
central seam function satisfies \(h_{\ell,0}\leq-\delta_0\) for some
\(\delta_0>0\).  The perturbation
\(\sum_{j=1}^{\nu-2}\eta_jy^j\) is uniformly small there, and the common
annulus estimate gives uniform continuous dependence of both the physical
and reference arcs on \((q,\lambda,\eta)\).  After shrinking these germs,
all such arcs satisfy \(h_{\ell,\eta}< -\delta_0/2\) outside the selected
box.  Thus every seam event belongs to that box, and its complement is a
smooth reference \(X^-_\lambda\)-transition.

Write \(\widetilde q\) for the incoming height in the prepared flow box.
On the selected ray, translate the radial section coordinate so that the
same prepared reference orbit has radial coordinate zero.  Smooth flow-box
dependence gives a jointly smooth orientation-preserving coordinate change
\begin{equation}
  q_{\mathrm{rad}}=\chi_{\lambda,\eta}(\widetilde q),
  \qquad
  \chi_{\lambda,\eta}(0)=0,
  \qquad
  \partial_{\widetilde q}\chi_{\lambda,\eta}(0)>0.
  \label{eq:closed-prepared-section-coordinate}
\end{equation}
On \(\eta=0\), the prepared height is the based first-integral coordinate
from \eqref{eq:closed-polynomial-input-factor}; hence
\begin{equation}
  \chi_{\lambda,0}=\rho_\lambda^{-1},
  \qquad
  \partial_{\widetilde q}\chi_{0,0}(0)=1.
  \label{eq:closed-prepared-section-coordinate-central}
\end{equation}
If \(P^{\mathrm{rad}}_{\lambda,\eta}\) denotes the time-\(2\pi\) first
return in the based radial coordinate, then the prepared return is exactly
\begin{equation}
  P^{\mathrm{prep}}_{\lambda,\eta}
  =\chi_{\lambda,\eta}^{-1}\circ
   P^{\mathrm{rad}}_{\lambda,\eta}\circ
   \chi_{\lambda,\eta}.
  \label{eq:closed-prepared-return-conjugacy}
\end{equation}
Thus it records the same first intersection and the same periodic orbits.
The transformation law
\eqref{eq:section-coordinate-coefficient-law} therefore preserves the sign
of the leading smooth coefficient, the rank-\(n\) reference jet,
simplicity, stability, and the sign of the product of the smooth and active
coefficients.  The outgoing reference derivative is \(L_0=1\).  In
particular, the reference \(X^-_\lambda\)-displacement and its rank are those of
\cref{thm:closed-polynomial-realization} in equivalent based coordinates,
while the transmitted coefficient has the sign of \(\beta\).  The common
full-turn argument at the start of this proof supplies the actual
first-return domain.  Hence
\cref{thm:physical-multiwell-cyclicity} applies and gives
\eqref{eq:closed-transverse-physical-bound}; the positive factor relating
its active coefficient to \(\beta\) makes the sign coupling
\(\sgn(\gamma\beta)\).  Restriction to \(\eta=0\) gives exactly the family
of \cref{thm:closed-polynomial-realization}, so its sharpness and stability
conclusions complete the proof.
\end{proof}

\bibliographystyle{unsrt}
\bibliography{references}

\end{document}